\documentclass[11pt]{article}

\usepackage[utf8]{inputenc} 
\usepackage{amsfonts}
\usepackage{amsmath}
\usepackage{amsthm} % not for journal version
\usepackage{color}
\usepackage{graphicx}
\usepackage{setspace}
\usepackage{url}
\usepackage{enumerate}
\usepackage{booktabs}
\usepackage{subfigure}

\definecolor{myred}{RGB}{255,50,50}         % revision version
\definecolor{myblack}{RGB}{0,0,0}           % normal version

\newtheorem{theorem}{Theorem}[section]

\newtheorem{lemma}[theorem]{Lemma}
\newtheorem{definition}[theorem]{Definition}

\newtheorem{proposition}[theorem]{Proposition}

\newtheorem{assumption}[theorem]{Assumption}

\numberwithin{equation}{section}

\newcommand{\grad}{\nabla}                       % gradient
\renewcommand{\implies}{\Rightarrow}             % implies
\newcommand{\inner}[2]{\langle#1,#2\rangle}      % inner product
\newcommand{\lin}{\mathrm{lin}}                  % linearity space
\newcommand{\norm}[1]{\|#1\|}                    % norm
\renewcommand{\Re}{\mathbb{R}}                   % real space
\renewcommand{\S}{\mathbb{S}}                    % symmetric matrices space
\newcommand{\T}{\top\hspace{-1pt}}               % transpose
\newcommand{\tr}{\mathrm{tr}}                    % trace	
\newcommand{\PS}{P_{\S^m_+}}                     % projection onto S^m_+
\renewcommand{\L}{\mathcal{L}}                   % Lyapunov operator
\newcommand{\diag}{\mathrm{diag}}  
\begin{document}

%=============================================================================

\title{Exact Augmented Lagrangian Functions for\\ Nonlinear
  Semidefinite Programming%
  \thanks{This is a pre-print of an article published in Computational
    Optimization and Applications. The final authenticated version is
    available online at:
    \protect\url{https://doi.org/10.1007/s10589-018-0017-z}.  This
    work was supported by the Grant-in-Aid for Young Scientists (B)
    (26730012) and for Scientific Research (B) (15H02968) from Japan
    Society for the Promotion of Science.}  }

\author{ Ellen H. Fukuda%
  \thanks{Department of Applied Mathematics and Physics, Graduate
    School of Informatics, Kyoto University, Kyoto \mbox{606--8501},
    Japan (\texttt{ellen@i.kyoto-u.ac.jp}).}  \and Bruno
  F. Louren\c{c}o%
  \thanks{Department of Mathematical Informatics, Graduate School of
    Information Science \& Technology, University of Tokyo,
    \mbox{Tokyo 113--8656}, Japan
    (\texttt{lourenco@mist.i.u-tokyo.ac.jp}).}
}

\date{June 20, 2018}

\maketitle

%=============================================================================

\begin{abstract}
  \noindent In this paper, we study augmented Lagrangian functions for
  nonlinear semidefinite programming (NSDP) problems with exactness
  properties. The term \emph{exact} is used in the sense that the
  penalty parameter can be taken appropriately, so a single
  minimization of the augmented Lagrangian recovers a solution of the
  original problem. This leads to \mbox{reformulations} of NSDP problems into
  unconstrained nonlinear programming ones. Here, we first establish a
  unified framework for constructing these exact functions,
  generalizing Di Pillo and Lucidi's work from 1996, that was aimed at
  solving nonlinear programming problems. Then, through our framework,
  we propose a practical augmented Lagrangian function for NSDP,
  proving that it is continuously differentiable and exact under the
  so-called nondegeneracy condition. We also present some
  preliminary numerical experiments.\\

\end{abstract}

%=============================================================================

\section{Introduction}

The following \emph{nonlinear semidefinite programming} (NSDP)
problem is considered:
\begin{equation}
  \label{eq:nsdp}
  \tag{NSDP}
  \begin{array}{ll}
    \underset{x}{\mbox{minimize}} & f(x) \\ 
    \mbox{subject to} & G(x) \in \S_{+}^m,
  \end{array}
\end{equation}
where $f \colon \Re^n \to \Re$ and $G \colon \Re^n \to \S^m$ are twice
continuously differentiable functions, $\S^m$ is the linear space of
all real symmetric matrices of dimension $m \times m$, and $\S^m_+$ is
the cone of all positive semidefinite matrices in $\S^m$. For
simplicity, here we do not take equality constraints into
consideration. The above problem extends the well-known nonlinear
programming (NLP) and the linear semidefinite programming (linear SDP)
problems. NLP and linear SDP models are certainly important, but they
may be insufficient in applications where more general constraints are
necessary.  In particular, in the recent literature, some applications
of NSDP are considered in different fields, such as control
theory~\cite{ANT03,FAN01}, structural optimization~\cite{KT06,KS04},
truss design problems~\cite{BJKNZ00}, and finance~\cite{KKW03}.
However, compared to NLP and linear SDP models, there are still few
methods available to solve NSDP, and the theory behind them requires
more investigation.

Some theoretical issues associated to NSDP, like optimality
conditions, are discussed in~\cite{BS00,For00,Jar12,LFF16,Sha97}.
There are, in fact, some methods for NSDPs proposed in the literature
such as primal-dual interior-point, augmented Lagrangian,
filter-based, sequential quadratic programming, and exact penalty
methods. Nevertheless, there are few implementations and, as far as we
know, only two general-purpose solvers are able to handle nonlinear
semidefinite constraints: PENLAB/PENNON \cite{FKS13} and NuOpt
\cite{YYH12}.  For a complete survey, see Yamashita and
Yabe~\cite{YY15}, and references therein.

Here, our main object of interest is the so-called \emph{augmented
  Lagrangian functions} and this work can be seen as a stepping stone
towards new algorithms for NSDPs. An augmented Lagrangian function is
basically the usual Lagrangian function with an additional term that
depends on a positive coefficient, called the \emph{penalty
  parameter}. When there exists an appropriate choice of the
parameter, such that a single minimization of the augmented Lagrangian
recovers a solution to the original problem, then we say that this
function is \emph{exact}. This is actually the same definition of the
so-called \emph{exact penalty} function. The difference is that an
exact augmented Lagrangian function is defined on the product space of
the problem's variables and the Lagrange multipliers, and an exact
penalty function is defined on the same space of the original
problem's variables. Both exact functions, which are also called
\emph{exact merit} functions, have been studied quite extensively when
the original problem is an~NLP.

The first proposed exact merit functions were nondifferentiable, and
the basic idea was to incorporate terms into the objective function
that penalize constraint violations.  However, unconstrained
minimization of nondifferentiable functions demands special methods,
and so, continuously differentiable exact functions were considered
subsequently. For NLPs, both exact penalty and exact augmented
Lagrangian functions were studied. The first one has an advantage of
having to deal with less variables, but it tends to have a more
complicated formula, because the information of the Lagrange
multipliers is, in some sense, hidden in the formula.  In most exact
penalty functions, this is done by using a function that estimates the
value of the Lagrange multipliers associated to a
point~\cite{DPG89}. The evaluation of this estimate is, however,
computationally expensive.

To overcome such a drawback, exact augmented Lagrangian functions can
be considered, with a price of increasing the number of variables. The
choice between these two types of exact functions depends, of course,
on the optimization problem at hand. So, exact augmented Lagrangian
functions were proposed in~\cite{DPG79} and~\cite{DPG82}, by Di Pillo
and Grippo for NLP problems with equality and inequality constraints,
respectively. They were further investigated
in~\cite{Ber82a,DPLLP03,DPL01,Luc88}, with additional theoretical
issues and schemes for box-constrained NLP problems. However, as far
as we know, there are no proposals for exact augmented Lagrangian
functions for more general conic constrained problems, in particular,
for~NSDP. The augmented Lagrangian function considered by Correa
  and Ram\'irez~\cite{CR04}, and Shapiro and Sun~\cite{SS04}, for
  example, is not exact.

In this paper, we introduce a continuously differentiable exact
augmented Lagrangian function for NSDP problems. We also give a
unified framework for constructing such functions. More precisely, we
propose a generalized augmented Lagrangian function for NSDP, and give
conditions for it to be exact. The main difference between the
  classical (and not exact) augmented Lagrangian and this exact
  version is the addition of a term, that we define in
  Section~\ref{sec:framework} as~$\gamma$. This is a continuously
  differentiable function defined in the product space of problem’s
  variables and the Lagrange multipliers, with key properties that
  guarantee the exactness of the augmented Lagrangian function. A
general framework with such $\gamma$ term was also given by Di
Pillo and Lucidi in~\cite{DPL96} for the NLP case. Besides the
optimization problem, a difference between~\cite{DPL96} and our work
is that, here, we propose the generalization first, and then construct
one particular exact augmented Lagrangian function. We believe that
the generalized function can be used in the future to easily build
other exact merit functions, together with possibly useful
methods. Meanwhile, we make some preliminary numerical
  experiments with the particular exact function, using a quasi-Newton
  method.

The paper is organized as follows. In Section~\ref{sec:preliminaries},
we start with basic definitions and necessary results associated to
NSDP problems. In Section~\ref{sec:framework}, a general framework for
constructing augmented Lagrangian with exactness properties is
given. A practical exact augmented Lagrangian as well as its exactness
results are given in Section~\ref{sec:exact_auglag}. This
  particular function is used in Section~\ref{sec:experiments}, where
  some numerical examples are presented. We conclude in
Section~\ref{sec:conclusions}, with some final remarks.

%=============================================================================

\section{Preliminaries}
\label{sec:preliminaries}

Let us first present some basic notations that will be used throughout
the paper. Let $x \in \Re^r$ be a $r$-dimensional column vector and $Z
\in \S^s$ a symmetric matrix with dimension $s \times s$. We use $x_i$
and $Z_{ij}$ to denote the $i$th element of $x$ and $(i,j)$ entry
($i$th row and $j$th column) of $Z$, respectively. We also use the
notation $[x_i]_{i=1}^r$ and $[Z_{ij}]_{i,j=1}^s$ to denote $x$ and
$Z$, respectively. The trace of $Z$ is denoted by $\tr(Z) :=
\sum_{i=1}^s Z_{ii}$. Moreover, if $Y \in \S^s$, then the inner
product of $Y$ and $Z$ is written as $\inner{Y}{Z} := \tr (Y Z)$, and
the Frobenius norm of $Z$ is given by $\norm{Z}_F :=
\inner{Z}{Z}^{1/2}$. The identity matrix, with dimension defined in
each context, is denoted by~$I$, and $\PS$ denotes the projection onto
the cone~$\S^m_+$.

For a function $p \colon \Re^s \to \Re$, its gradient and Hessian at a
point $x \in \Re^s$ are given by $\grad p(x) \in \Re^s$ and $\grad^2
p(x) \in \Re^{s \times s}$, respectively. For $q \colon \S^{\ell} \to
\Re$, $\grad q(Z)$ denotes the matrix with $(i,j)$ term given by the
partial derivatives $\partial q(Z)/\partial Z_{ij}$. If $\psi \colon
\Re^s \times \S^{\ell} \to \Re$, then its gradient at $(x,Z) \in \Re^s
\times \S^{\ell}$ with respect to $x$ and {$Z$} are denoted by
$\grad_x \psi(x,{Z}) \in \Re^s$ and $\grad_{{Z}}
\psi(x,{Z}) \in \S^{\ell}$, respectively. Similarly, the Hessian
of $\psi$ at $(x,Z)$ with respect to $x$ is written as $\grad_{xx}^2
\psi(x,{Z})$. For any linear operator $\mathcal{G} \colon \Re^s \to
\S^\ell$ defined by $\mathcal{G}v = \sum_{i=1}^s v_i \mathcal{G}_i$
with $\mathcal{G}_i \in \S^\ell$, $i=1,\dots,s$, and $v \in \Re^s$,
the adjoint operator $\mathcal{G}^*$ is defined by
\begin{displaymath}
  \mathcal{G}^* Z = (\inner{\mathcal{G}_1}{Z}, \dots,
  \inner{\mathcal{G}_s}{Z})^\T, \quad Z \in \S^\ell.
\end{displaymath}
Given a mapping $\mathcal{G} \colon \Re^s \to \S^\ell$, its derivative
at a point $x \in \Re^s$ is denoted by $\grad \mathcal{G}(x) \colon
\Re^s \to \S^\ell$ and defined by
\begin{displaymath}
  \grad \mathcal{G}(x) v = \sum_{i=1}^s v_i \frac{\partial
    \mathcal{G}(x)}{\partial x_i}, \quad v \in \Re^s,
\end{displaymath}
where $\partial \mathcal{G}(x)/\partial x_i \in \S^\ell$ are the
partial derivative matrices.

One important operator that is necessary when dealing with NSDP
problems is the \emph{Jordan product} associated to the
space~$\S^m$. For any $Y, Z \in \S^m$, it is defined by
\begin{displaymath}
  Y \circ Z := \frac{YZ + ZY}{2}.
\end{displaymath}
Taking $Y \in \S^m$, we also denote by $\L_Y \colon \S^m \to
\S^m$ the linear operator given by
\begin{displaymath}
  \L_Y(Z) : = Y \circ Z.
\end{displaymath}
Since we are only considering the space $\S^m$ of symmetric matrices,
we have $\L_Y(Z) = \L_Z(Y)$. In the following lemmas, we present some
useful results associated to this Jordan product and the projection
operator~$\PS$.

\begin{lemma}
  \label{lem:properties1}
  For any matrix $Z \in \Re^{m \times m}$, the following statements
  hold:
  \begin{itemize}
  \item[(a)] $\PS(-Z) = \PS(Z) - Z$;
  \item[(b)] $\PS(Z) \circ \PS(-Z) = 0$.
  \end{itemize}
\end{lemma}

\begin{proof}
  See \cite[Section~1]{Tse98}.
\end{proof}

\begin{lemma}
  \label{lem:properties2}
  If $Y, Z \in \S^m$, then the following statements are equivalent:
  \begin{itemize}
  \item[(a)] $Y, Z \in \S^m_+$ and $Y \circ Z = 0$; 
  \item[(b)] $Y, Z \in \S^m_+$ and $\inner{Y}{Z} = 0$; 
  \item[(c)] $Y - \PS(Y-Z) = 0$.
  \end{itemize}
\end{lemma}

\begin{proof}
  It follows from \cite[Section~8.12]{Ber09} and
  \cite[Lemma~2.1(b)]{Tse98}.
\end{proof}

\begin{lemma}
  \label{lem:derivatives}
  The following statements hold.
  \begin{itemize}
  \item[(a)] Let $Q \colon \Re^n \to \S^m$ be a differentiable
    function, and define $\psi \colon \Re^n \to \Re$ as $\psi(x) :=
    \norm{\PS(Q(x))}^2_F$. Then, the gradient of $\psi$ at $x \in
    \Re^n$ is given by
    \begin{displaymath}
      \grad \psi(x) = 2 \grad Q(x)^* \PS (Q(x)).
    \end{displaymath}
    A similar result holds when the domain of the functions~$Q$ and
    $\psi$ is changed to~$\S^m$.
  \item[(b)] Let $R_1,R_2 \colon \Re^n \to \S^m$ be differentiable
    functions, and define $P \colon \Re^n \to \S^m$ as $P(x) :=
    \L_{R_1(x)}(R_2(x)) = R1(x) \circ R_2(x)$. Then, we have
    \begin{displaymath}
      \grad P(x)^* Z = \displaystyle{\left[ \left\langle
            \frac{\partial R_1(x)}{\partial x_i} \circ R_2(x) + R_1(x) \circ
            \frac{\partial R_2(x)}{\partial x_i}, Z \right\rangle
        \right]_{i=1}^n} \:\: \mbox{for all } Z \in \S^m.
    \end{displaymath}
  \item[(c)] Let $\xi \colon \Re^n \to \Re$ be a differentiable
    function, and define $S \colon \Re^n \to \S^m$ as $S(x) :=
    \xi(x)W$, with $W \in \S^m$. Then, we obtain
    \begin{displaymath}
      \grad S(x)^* Z = \inner{W}{Z} \grad \xi(x) \quad \mbox{for all } 
      Z \in \S^m.
    \end{displaymath}
  \item[(d)] Let $\eta \colon \S^m \to \Re$ be a differentiable
    function, and define $T \colon \S^m \to \S^m$ as $T(Y) := \eta(Y)
    Y$. Then, we have
    \begin{displaymath}
      \grad T(Y)^*Z = \inner{Y}{Z} \grad \eta(Y) + \eta(Y) Z \quad
      \mbox{for all } Z \in \S^m.
    \end{displaymath}
  \end{itemize}
\end{lemma}

\begin{proof}
  Item~(a) follows from \cite[Corollary~3.2]{Lew96} and item~(b)
  follows easily from the definitions of adjoint operator and Jordan
  product. Item~(c) holds also from the definition of adjoint
  operator, and because $\partial S(x) / \partial x_i = (\partial
  \xi(x) / \partial x_i) W$ for all~$i$. For item~(d), observe that
  for all $W \in \S^m$, we obtain
  \begin{eqnarray*}
    \grad T(Y)(W) & = & \lim_{t \downarrow 0} \frac{T(Y+tW) - T(Y)}{t}\\
    & = & \lim_{t \downarrow 0} \frac{\eta(Y+tW) (Y+tW) - \eta(Y)Y}{t}\\
    & = & \lim_{t \downarrow 0} \frac{\big( \eta(Y+tW) - \eta(Y) \big)Y}{t} 
    + \eta(Y+tW)W \\
    & = & \eta'(Y;W)Y + \eta(Y)W,
  \end{eqnarray*}
  where $\eta'(Y;W)$ is the directional derivative of~$\eta$ at~$Y$ in
  the direction~$W$. From the differentiability of~$\eta$, we have
  $\grad T(Y)(W) = \inner{\grad \eta(Y)}{W} Y + \eta(Y) W$. Recalling that $\grad T(Y)^*$ denotes the adjoint of $\grad T(Y)$, this
  equality yields
  \begin{displaymath}
    \inner{W}{\grad T(Y)^* Z} = \inner{\grad T(Y)(W)}{Z} =
    \inner{\grad \eta(Y)}{W} \inner{Y}{Z} + \eta(Y) \inner{W}{Z}
  \end{displaymath}
  for all $W,Z \in \S^m$, which completes the proof.
  % For item~(b), observe that $\grad
  % \norm{\PS(S(x))}^2_F = 2 \grad S(x)^* \PS (S(x))$ for all
  % $x$. Moreover, $\partial S(x)/\partial x_i = \psi(x) (\partial Q(x)
  % / \partial x_i) + (\partial \psi(x)/\partial x_i) Q(x)$ also
  % holds. Thus, the result follows once again from the definition of
  % adjoint operator.
\end{proof}

Let us return to problem~\eqref{eq:nsdp}. Define $L \colon \Re^n
\times \S^m \to \Re$ as the Lagrangian function
associated to problem~\eqref{eq:nsdp}, that~is,
\begin{displaymath}
  L(x,\Lambda) := f(x) - \inner{G(x)}{\Lambda}.
\end{displaymath}
The pair $(x,\Lambda) \in \Re^n \times \S^m$ satisfies the \emph{KKT
  conditions} of problem~\eqref{eq:nsdp} (or, it is a KKT pair) if the
following conditions hold:
\begin{equation}
  \label{eq:nsdp_kkt}
  \begin{array}{rcl}
    \grad_x L(x,\Lambda) & = & 0, \\
    \Lambda \circ G(x) & = & 0, \\
    G(x) & \in & \S_{+}^m, \\ 
    \Lambda & \in & \S_{+}^m, \\ 
   \end{array}
\end{equation}
where
\begin{displaymath}
  \grad_x L(x,\Lambda) = \grad f(x) - \grad G(x)^* \Lambda.
\end{displaymath}
The above conditions are necessary for optimality under a constraint
qualification. Moreover, Lemma~\ref{lem:properties2} shows that the
condition $\Lambda \circ G(x) = 0$ can be replaced by
$\inner{\Lambda}{G(x)} = 0$ because $G(x) \in \S_{+}^m$ and $\Lambda
\in \S_{+}^m$ hold. Furthermore, it can be shown that this condition
can also be replaced by $\Lambda G(x) = 0$~\cite[Section~2]{YY15}.

Now, consider the nonlinear programming below:
\begin{equation}
  \label{eq:unc_nlp}
  \begin{array}{ll}
    \underset{x,\Lambda}{\mbox{minimize}} & \Psi_c(x,\Lambda) \\ 
    \mbox{subject to} & (x,\Lambda) \in \Re^n \times \S^m,
  \end{array}
\end{equation}
where $\Psi_c \colon \Re^n \times \S^m \to \Re$, and $c > 0$ is a
penalty parameter. Observe that the above problem is unconstrained,
with both the original variable~$x$ and the Lagrange
multiplier~$\Lambda$ as variables. As usual, we say that~$(x,\Lambda)$
is \emph{stationary} of~$\Psi_c$ (or for problem~\eqref{eq:unc_nlp})
when $\grad \Psi_c(x,\Lambda) = 0$. We use~$G_{\mbox{\tiny{NLP}}}(c)$
and~$L_{\mbox{\tiny{NLP}}}(c)$ to denote the sets of global and local
minimizers, respectively, of problem~\eqref{eq:unc_nlp}. We also
define~$G_{\mbox{\tiny{NSDP}}}$ and~$L_{\mbox{\tiny{NSDP}}}$ as the
set of global and local minimizers of problem~\eqref{eq:nsdp},
respectively. Using such notations, we present the formal definition
of exact augmented Lagrangian functions.

\begin{definition}
  \label{def:exactaug}
  A function $\Psi_c \colon \Re^n \times \S^m \to \Re$ is
  called an \emph{exact augmented Lagrangian} function associated
  to~\eqref{eq:nsdp} if, and only if, there exists $\hat{c} > 0$
  satisfying the following:
  \begin{itemize}
  \item[(a)] For all $c \ge \hat{c}$, if $(x,\Lambda) \in
    G_{\mbox{\tiny{NLP}}}(c)$, then $x \in
    G_{\mbox{\tiny{NSDP}}}$ and $\Lambda$ is a corresponding
    Lagrange multiplier. Conversely, if $x \in
    G_{\mbox{\tiny{NSDP}}}$ with $\Lambda$ as a corresponding
    Lagrange multiplier, then $(x,\Lambda) \in
    G_{\mbox{\tiny{NLP}}}(c)$ for all $c \ge \hat{c}$.
  \item[(b)] For all $c \ge \hat{c}$, if $(x,\Lambda) \in
    L_{\mbox{\tiny{NLP}}}(c)$, then $x \in
    L_{\mbox{\tiny{NSDP}}}$ and $\Lambda$ is a corresponding
    Lagrange multiplier.
  \end{itemize}
\end{definition}

Basically, the above definition shows that $\Psi_c$ is an exact
augmented Lagrangian function when, without considering Lagrange
multipliers, there are equivalence between the global minimizers, and
if all local solutions of~\eqref{eq:unc_nlp} are local solutions
of~\eqref{eq:nsdp}, for penalty parameters greater than a threshold
value. It means that the original constrained conic
problem~\eqref{eq:nsdp} can be replaced with an unconstrained nonlinear
programming problem~\eqref{eq:unc_nlp} when the penalty parameter is
chosen appropriately. Note that the definition of exact penalty
functions is similar. The only difference is that in the exact penalty
case, the objective function of problem~\eqref{eq:unc_nlp} does not
involve Lagrange multipliers explicitly.

%=============================================================================

\section{A general framework}
\label{sec:framework}

In this section, we propose a general formula for continuously
differentiable augmented Lagrangian functions associated to NSDP
problems, with exactness properties. It can be seen as a
generalization of the one proposed by Di Pillo and Lucidi
in~\cite{DPL96} for NLP problems. With this purpose, let us first
define the following function $\varphi \colon \S^m \times \S^m \to
\Re$:
\begin{equation}
  \label{eq:phi}
  \varphi(Y,Z) := \left\| \PS \left( \frac{Z}{2} - Y \right) \right\|_F^2
  - \frac{\norm{Z}^2_F}{4}.
\end{equation}
Observe that this function is continuously differentiable because
$\norm{\cdot}^2_F$ and $\norm{\PS(\cdot)}^2_F$ are both continuously
differentiable. Moreover, it has the properties below.

\begin{lemma}
  \label{lem:Phi}
  Let $\varphi \colon \S^m \times \S^m \to \Re$ be defined
  by~\eqref{eq:phi}. Then, the following statements hold.
  \begin{itemize}
  \item[(a)] If $Y,Z \in \S^m_+$ and $\inner{Y}{Z} = 0$, then
    $\varphi(Y,Z) = 0$.
  \item[(b)] If $Y \in \S^m_+$, then $\varphi(Y,Z) \le 0$ for all $Z \in
    \S^m$.
  \end{itemize}
\end{lemma}

\begin{proof}
  (a) Clearly, $Z/2 \in \S^m_+$ because $\S^m_+$ is a cone. From
  Lemma~\ref{lem:properties2}, we have $Z/2 = \PS (Z/2 - Y)$. Thus,
  taking the square of the Frobenius norm in both sides of this
  expression gives the result.\\

  \noindent (b) Since $Y \in \S^m_+$, we obtain $\PS(-Y) = 0$. Using
  this fact and the nonexpansive property of the projection, we
  get
  \begin{displaymath}
    \left\| \PS \left( \frac{Z}{2} - Y \right) \right\|_F = 
    \left\| \PS \left( \frac{Z}{2} - Y \right) - 
      \PS \left( -Y \right) \right\|_F
    \le \left\| \frac{Z}{2} \right\|_F.
  \end{displaymath}
  Thus, the result follows by squaring both sides of the above inequality.
\end{proof}

We propose a \emph{generalized augmented Lagrangian function}
$\mathcal{A}_c \colon \Re^n \times \S^m \to \Re$ as follows:
\begin{equation}
  \label{eq:A_c}
  \mathcal{A}_c(x,\Lambda) := f(x) + \alpha_c(x,\Lambda) 
  \varphi \big( G(x),\beta_c(x,\Lambda)\Lambda \big) + \gamma(x,\Lambda),
\end{equation}
where $c > 0$ is a penalty parameter, $\alpha_c, \beta_c, \gamma
\colon \Re^n \times \S^m \to \Re$, and $\varphi$ is given
in~\eqref{eq:phi}, namely
\begin{displaymath}
  \varphi \big( G(x),\beta_c(x,\Lambda)\Lambda \big) = \left\| \PS
  \left( \frac{\beta_c(x,\Lambda)}{2} \Lambda - G(x) \right)
  \right\|_F^2 - \frac{\beta_c(x,\Lambda)^2}{4} \norm{\Lambda}^2_F.
\end{displaymath} 
We will show now that~$\mathcal{A}_c$ is an exact augmented Lagrangian
function associated to~\eqref{eq:nsdp} in the sense of
Definition~\ref{def:exactaug}, when certain assumptions for
$\alpha_c$, $\beta_c$, and $\gamma$ are satisfied.

\begin{assumption}
  \label{hip:general_func}
  The functions $\alpha_c, \beta_c, \gamma \colon \Re^n \times \S^m
  \to \Re$ satisfy the following conditions.
  \begin{itemize}
  \item[(a)] $\alpha_c, \beta_c, \gamma$ are continuously
    differentiable for all $c > 0$.
  \item[(b)] $\alpha_c(x,\Lambda) > 0$ for all $x$ feasible
    for~\eqref{eq:nsdp}, $\Lambda \in \S^m$, and all $c > 0$.
  \end{itemize}
  Moreover, if $(\bar{x},\bar{\Lambda}) \in \Re^n \times \S^m$ is a
  KKT pair of~\eqref{eq:nsdp}, then the conditions below hold.
  \begin{itemize}
  \item[(c)] $\alpha_c(\bar{x},\bar{\Lambda})
    \beta_c(\bar{x},\bar{\Lambda}) = 1$ for all $c > 0$.
  \item[(d)] $\gamma(\bar{x},\bar{\Lambda}) = 0$, $\grad_x
    \gamma(\bar{x},\bar{\Lambda}) = 0$, and $\grad_{\Lambda}
    \gamma(\bar{x},\bar{\Lambda}) = 0$.
  \item[(e)] There exist neighborhoods $V_{\bar{x}}$ and
    $V_{\bar{\Lambda}}$ of $\bar{x}$ and $\bar{\Lambda}$,
    respectively, and a continuous function $\Gamma \colon V_{\bar{x}}
    \to V_{\bar{\Lambda}}$ such that $\Gamma(\bar{x}) = \bar{\Lambda}$
    and $\gamma(x,\Gamma(x)) = 0$ for all $x \in V_{\bar{x}}$.
  \end{itemize}
\end{assumption}

\begin{proposition}
  \label{prop:grad_A_c}
  Suppose that Assumption~\ref{hip:general_func}(a) holds. Then, the
  function $\mathcal{A}_c$ defined in~\eqref{eq:A_c} is continuously
  differentiable. Moreover, its gradient with respect to $x$ and
  $\Lambda$, respectively, can be written as follows:
  \begin{eqnarray*}
    \grad_x \mathcal{A}_c(x,\Lambda) & = & \grad f(x)  
    + \varphi \big( G(x),\beta_c(x,\Lambda)\Lambda \big) 
    \grad_x \alpha_c(x,\Lambda)\\
    & & {} - 2 \alpha_c(x,\Lambda) \grad G(x)^* 
    \PS \left(  \frac{\beta_c(x,\Lambda)}{2} \Lambda - G(x) \right)\\
    & & {} + \alpha_c(x,\Lambda) \left\langle \Lambda,
      \PS \left( \frac{\beta_c(x,\Lambda)}{2} \Lambda - G(x) \right)
      \right\rangle \grad_x \beta_c(x,\Lambda)\\
    & & {} - \frac{1}{2} \alpha_c(x,\Lambda) \beta_c(x,\Lambda) 
    \norm{\Lambda}_F^2 \grad_x \beta_c(x,\Lambda) 
    + \grad_x \gamma(x,\Lambda),\\[5pt]
    \grad_\Lambda \mathcal{A}_c (x,\Lambda) & = & 
    \varphi \big( G(x), \beta_c(x,\Lambda) \Lambda \big) 
    \grad_\Lambda \alpha_c(x,\Lambda) \\ 
    & & {} + \alpha_c(x,\Lambda) \Bigg[ \left\langle \Lambda,
      \PS \left( \frac{\beta_c(x,\Lambda)}{2} \Lambda - G(x) \right)
    \right\rangle \grad_{\Lambda} \beta_c(x,\Lambda) \\
    & & \quad + \beta_c(x,\Lambda) \PS \left( \frac{\beta_c(x,\Lambda)}{2} 
      \Lambda - G(x) \right)\\
    & & \quad - \frac{1}{2} \beta_c(x,\Lambda) \norm{\Lambda}_F^2
    \grad_{\Lambda}  \beta_c(x,\Lambda) - \frac{1}{2} 
    \beta_c(x,\Lambda)^2 \Lambda \Bigg] + \grad_\Lambda \gamma(x,\Lambda).
  \end{eqnarray*}
\end{proposition}

\begin{proof}
  The continuous differentiability of $\mathcal{A}_c$ follows from
  Assumption~\ref{hip:general_func}(a) and the fact that $f$, $G$, and
  $\varphi$ are continuously differentiable. For the gradient's
  formula, we use Lemma~\ref{lem:derivatives}(a),(c),(d) and some
  simple calculations.
\end{proof}

Before proving the exactness results, we will first show the relation
between the function~$\mathcal{A}_c$ and the objective function~$f$
of~\eqref{eq:nsdp}. As we can see in the next propositions, the values
of~$\mathcal{A}_c$ and~$f$ at KKT points coincide, but if a point is
only feasible, then a simple inequality holds.

\begin{proposition}
  \label{prop:A_c_feasible}
  Suppose that Assumption~\ref{hip:general_func}(b) holds. Let $x \in
  \Re^n$ be a feasible point of~\eqref{eq:nsdp}. Then,
  $\mathcal{A}_c(x,\Lambda) \le f(x) + \gamma(x,\Lambda)$ for all
  $\Lambda \in \S^m$ and all $c > 0$.
\end{proposition}

\begin{proof}
  Let $\Lambda \in \S^m$ and $c > 0$ be taken arbitrarily. Since $x$
  is feasible for~\eqref{eq:nsdp}, we have $G(x) \in \S^m_+$. Thus,
  Lemma~\ref{lem:Phi}(b) shows that $\varphi (G(x),\beta_c(x,\Lambda)
  \Lambda) \le 0$ is satisfied. The proof is complete because
  $\alpha_c(x,\Lambda) > 0$ also holds from
  Assumption~\ref{hip:general_func}(b).
\end{proof}

\begin{proposition}
  \label{prop:nsdp2A_c}
  Suppose that Assumption~\ref{hip:general_func} holds.  Let
  $(x,\Lambda) \in \Re^n \times \S^m$ be a KKT pair
  of~\eqref{eq:nsdp}. Then, $(x,\Lambda)$ is also
  stationary of $\mathcal{A}_c$, and
  $\mathcal{A}_c(x,\Lambda) = f(x)$ for all $c > 0$.
\end{proposition}

\begin{proof}
  Let $c > 0$ be arbitrarily given and recall the formulas of $\grad_x
  \mathcal{A}_c(x,\Lambda)$ and $\grad_\Lambda
  \mathcal{A}_c(x,\Lambda)$ given in
  Proposition~\ref{prop:grad_A_c}. From
  Assumption~\ref{hip:general_func}(b),(c), we have
  $\beta_c(x,\Lambda) > 0$. So, from the KKT
  conditions~\eqref{eq:nsdp_kkt}, $G(x) \in \S^m_+$,
  $\beta_c(x,\Lambda) \Lambda \in \S^m_+$, and
  $\inner{G(x)}{\beta_c(x,\Lambda)\Lambda} =
  0$ also hold, which imply that
  \begin{equation}
    \label{eq:prop:nsdp2A_c.1}
    \varphi \big( G(x),\beta_c(x,\Lambda) \Lambda 
    \big) = 0 
  \end{equation}
  from Lemma~\ref{lem:Phi}(a). Moreover, Lemma~\ref{lem:properties2}
  shows that
  \begin{equation}
    \label{eq:prop:nsdp2A_c.2}
    \PS \left( \frac{\beta_c(x,\Lambda)}{2} \Lambda -  
      G(x) \right) =  \frac{\beta_c(x,\Lambda)}{2} 
    \Lambda.
  \end{equation}
  The equalities~\eqref{eq:prop:nsdp2A_c.1},
  \eqref{eq:prop:nsdp2A_c.2} and Assumption~\ref{hip:general_func}(d)
  yield $\grad_\Lambda \mathcal{A}_c (x,\Lambda) =
  0$. Moreover, from~\eqref{eq:prop:nsdp2A_c.2} and
  Assumption~\ref{hip:general_func}(c), we have
  \begin{displaymath}
    \grad f(x) - 2 \alpha_c(x,\Lambda) \grad G(x)^* 
    \PS \left( \frac{\beta_c(x,\Lambda)}{2} \Lambda 
      - G(x) \right) = \grad f(x) 
    - \grad G(x)^* \Lambda.
  \end{displaymath}
  So, once again using Assumption~\ref{hip:general_func}(d),
  equalities~\eqref{eq:prop:nsdp2A_c.1}, \eqref{eq:prop:nsdp2A_c.2}
  and the KKT condition $\grad f(x) - \grad G(x)^*
  \Lambda = 0$, we can conclude that $\grad_x
  \mathcal{A}_c(x,\Lambda) = 0$ holds. Finally,
  \eqref{eq:prop:nsdp2A_c.1} and Assumption~\ref{hip:general_func}(d)
  also yields $\mathcal{A}_c(x,\Lambda) = f(x)$, and
  the proof is complete.
\end{proof}

The above proposition shows that a KKT pair of~\eqref{eq:nsdp} is
stationary of~$\mathcal{A}_c$, and this assertion does not depend on
the parameter~$c$. The exactness properties of~$\mathcal{A}_c$ can be
shown only if the other implication also holds, that is, a stationary
point of~$\mathcal{A}_c$ should be a KKT pair of~\eqref{eq:nsdp}, at
least when $c$ is greater than some threshold value. If such a
statement holds, then the exactness of $\mathcal{A}_c$ is guaranteed,
as it can be seen below. Before that, we recall that the global
(local) minimizers of problems~\eqref{eq:nsdp} and~\eqref{eq:unc_nlp}
with $\Psi_c := \mathcal{A}_c$ are, respectively, denoted by
$G_{\mbox{\tiny{NSDP}}}$ ($L_{\mbox{\tiny{NSDP}}}$) and
$G_{\mbox{\tiny{NLP}}}(c)$ ($L_{\mbox{\tiny{NLP}}}(c)$). We also
consider the following assumption:
\begin{assumption}
  \label{hip:cq}
  The sets $G_{\mbox{\tiny{NSDP}}}$ and $G_{\mbox{\tiny{NLP}}}(c)$ are
  nonempty for all $c > 0$. Moreover, for every $x \in
  G_{\mbox{\tiny{NSDP}}}$ there is at least one $\Lambda \in \S^m$ such that 
  $(x,\Lambda)$ is a KKT pair of \eqref{eq:nsdp}.
\end{assumption}

\noindent The existence of optimal solutions of the unconstrained
problem is guaranteed if an extraneous compact set is
considered~\cite{DPG89}, or by exploiting some properties of the
problem, as coercivity and monotonicity~\cite{AS10}. Furthermore, 
we can ensure the existence of a Lagrange multiplier by imposing 
some constraint qualification. Now, if we
define
\begin{eqnarray}
  \tilde{G}_{\mbox{\tiny{NSDP}}} & := & \big\{ (x,\Lambda) \colon 
  x \in G_{\mbox{\tiny{NSDP}}} \mbox{ and } \label{eq:tilde_G_NSDP} \\
  & & \quad \Lambda \mbox{ is a corresponding Lagrange multiplier} \big\},
  \nonumber
\end{eqnarray}
then, using Assumption~\ref{hip:cq}, we obtain
$\tilde{G}_{\mbox{\tiny{NSDP}}} \ne \emptyset$.

\begin{lemma}
  \label{lem:global_subset}
  Suppose that Assumptions~\ref{hip:general_func} and~\ref{hip:cq}
  hold. Then for all $c > 0$, $G_{\mbox{\tiny{NLP}}}(c) \subseteq
  \tilde{G}_{\mbox{\tiny{NSDP}}}$ implies $G_{\mbox{\tiny{NLP}}}(c) =
  \tilde{G}_{\mbox{\tiny{NSDP}}}$.
\end{lemma}

\begin{proof}
  Let $c > 0$ be arbitrarily given and $(\bar{x},\bar{\Lambda}) \in
  G_{\mbox{\tiny{NLP}}}(c)$. By assumption, $(\bar{x},\bar{\Lambda})
  \in \tilde{G}_{\mbox{\tiny{NSDP}}}$ also holds, and thus,
  $(\bar{x},\bar{\Lambda})$ is a KKT pair of~\eqref{eq:nsdp}. From
  Proposition~\ref{prop:nsdp2A_c}, we obtain
  \begin{equation}
    \label{eq:global_subset.1}
    f(\bar{x}) = \mathcal{A}_c(\bar{x},\bar{\Lambda}). 
  \end{equation}
  Recall that $\tilde{G}_{\mbox{\tiny{NSDP}}} \ne \emptyset$ because
  of Assumption~\ref{hip:cq}. So, take $(\tilde{x},\tilde{\Lambda})
  \in \tilde{G}_{\mbox{\tiny{NSDP}}}$, with
  $(\tilde{x},\tilde{\Lambda}) \ne (\bar{x},\bar{\Lambda})$.  Since
  $(\tilde{x},\tilde{\Lambda})$ satisfies the KKT conditions
  of~\eqref{eq:nsdp}, once again from Proposition~\ref{prop:nsdp2A_c},
  we have $f(\tilde{x}) = \mathcal{A}_c(\tilde{x},\tilde{\Lambda})$.
  This fact, together with~\eqref{eq:global_subset.1}, and the
  definition of global solutions, gives
  \begin{displaymath}
    \mathcal{A}_c(\tilde{x},\tilde{\Lambda}) = f(\tilde{x}) \le
    f(\bar{x}) = \mathcal{A}_c(\bar{x},\bar{\Lambda}) \le
    \mathcal{A}_c(\tilde{x},\tilde{\Lambda}),
  \end{displaymath}
  which shows that the whole expression above holds with equalities.
  Therefore, $(\tilde{x},\tilde{\Lambda}) \in
  G_{\mbox{\tiny{NLP}}}(c)$, which completes the proof.
\end{proof}

\begin{theorem}
  \label{the:global_local}
  Suppose that Assumption~\ref{hip:general_func} and~\ref{hip:cq}
  hold. Assume also that there exists $\hat{c} > 0$ such that every
  stationary point of $\mathcal{A}_c$ is also a KKT pair
  of~\eqref{eq:nsdp} for all $c \ge \hat{c}$. Then, $\mathcal{A}_c$ is
  an exact augmented Lagrangian function associated
  to~\eqref{eq:nsdp}, in other words:
  \begin{itemize}
  \item[(a)] $G_{\mbox{\tiny{NLP}}}(c) = \tilde{G}_{\mbox{\tiny{NSDP}}}$
    for all $c \ge \hat{c}$.
  \item[(b)] $L_{\mbox{\tiny{NLP}}}(c) \subseteq \big\{ (x,\Lambda) \colon
    x \in L_{\mbox{\tiny{NSDP}}} \mbox{ and } \Lambda \mbox{ is a
      corresponding multiplier} \big\}$ for all $c \ge \hat{c}$.
  \end{itemize}
\end{theorem}

\begin{proof}
  (a) Let $c \ge \hat{c}$ be arbitrarily given. From
  Lemma~\ref{lem:global_subset}, we only need to prove that
  $G_{\mbox{\tiny{NLP}}}(c) \subseteq \tilde{G}_{\mbox{\tiny{NSDP}}}$.
  Let $(\bar{x},\bar{\Lambda}) \in
  G_{\mbox{\tiny{NLP}}}(c)$. From~\eqref{eq:tilde_G_NSDP}, we
    need to show that $\bar{x} \in G_{\mbox{\tiny{NSDP}}}$,
    with $\bar{\Lambda}$ as a corresponding Lagrange multiplier.
  Then, $(\bar{x},\bar{\Lambda})$ is a stationary point
  of~$\mathcal{A}_c$.  From this theorem's assumption, it is also a
  KKT pair of~\eqref{eq:nsdp}, which implies $f(\bar{x}) =
  \mathcal{A}_c(\bar{x},\bar{\Lambda})$ from
  Proposition~\ref{prop:nsdp2A_c}. Now, assume that there exists
  $\tilde{x} \in G_{\mbox{\tiny{NSDP}}}$ such that $\tilde{x} \ne
  \bar{x}$. Since~$\tilde{x}$ satisfies a constraint qualification
  from Assumption~\ref{hip:cq}, there exists $\tilde{\Lambda}$ such
  that $(\tilde{x},\tilde{\Lambda})$ satisfies the KKT conditions
  of~\eqref{eq:nsdp}. Once again by Proposition~\ref{prop:nsdp2A_c},
  we have $f(\tilde{x}) = \mathcal{A}_c(\tilde{x},\tilde{\Lambda})$.
  So, the definition of global minimizers gives
  \begin{displaymath}
    \mathcal{A}_c(\tilde{x},\tilde{\Lambda}) = f(\tilde{x}) \le
    f(\bar{x}) = \mathcal{A}_c(\bar{x},\bar{\Lambda}) \le
    \mathcal{A}_c(\tilde{x},\tilde{\Lambda}),
  \end{displaymath}
  which shows that the whole expression above holds with equalities.
  Therefore, $\bar{x} \in G_{\mbox{\tiny{NSDP}}}$, with
  $\bar{\Lambda}$ as a corresponding Lagrange multiplier.\\

  \noindent (b) Let $c \ge \hat{c}$ and $(\bar{x},\bar{\Lambda}) \in
  L_{\mbox{\tiny{NLP}}}(c)$. Since $(\bar{x},\bar{\Lambda})$ is a
  stationary point of~$\mathcal{A}_c$, it is also a KKT pair
  of~\eqref{eq:nsdp} by theorem's assumption. So, from
  Proposition~\ref{prop:nsdp2A_c}, we have
  \begin{equation}
    \label{eq:global_local.2}
    \mathcal{A}_c(\bar{x},\bar{\Lambda}) = f(\bar{x}).
  \end{equation}
  Moreover, from the definition of local minimizer, there exist
  neighborhoods $V_{\bar{x}}$ and $V_{\bar{\Lambda}}$ of $\bar{x}$ and
  $\bar{\Lambda}$, respectively, such that
  \begin{displaymath}
    \mathcal{A}_c(\bar{x}, \bar{\Lambda}) \le \mathcal{A}_c (x,\Lambda)
    \quad \mbox{for all } (x,\Lambda) \in V_{\bar{x}} \times V_{\bar{\Lambda}}.
  \end{displaymath}
  Here, we suppose that $V_{\bar{x}}$ and $V_{\bar{\Lambda}}$ are
  sufficiently small, which guarantees the existence of a function
  $\Gamma$ as in Assumption~\ref{hip:general_func}(e). In particular,
  we obtain $\mathcal{A}_c(\bar{x}, \bar{\Lambda}) \le \mathcal{A}_c
  (x,\Gamma(x))$ for all $x \in V_{\bar{x}}$. This inequality,
  together with~\eqref{eq:global_local.2}, shows that
  \begin{equation}
    f(\bar{x}) \le \mathcal{A}_c (x,\Gamma(x))
    \quad \mbox{for all } x \in V_{\bar{x}}.
  \end{equation}
  Thus, from Proposition~\ref{prop:A_c_feasible} and
  Assumption~\ref{hip:general_func}(e), we get
  \begin{displaymath}
    f(\bar{x}) \le \mathcal{A}_c (x,\Gamma(x)) \le
    f(x) + \gamma(x,\Gamma(x)) = f(x)
  \end{displaymath}
  for all $x \in V_{\bar{x}}$ that is feasible for~\eqref{eq:nsdp}.
  So, we conclude that $\bar{x} \in L_{\mbox{\tiny{NSDP}}}$.
\end{proof}

The above result shows that the generalized function~$\mathcal{A}_c$ is
an exact augmented Lagrangian function if a finite penalty parameter
$\hat{c} > 0$ satisfying
\begin{displaymath}
  \mbox{``stationary of } \mathcal{A}_c \mbox{ with parameter } c 
  \, \Longrightarrow \, \mbox{KKT of \eqref{eq:nsdp}''} \quad
    \mbox{for all } c \ge \bar{c}
\end{displaymath}
is guaranteed to exist. However, even if
Assumption~\ref{hip:general_func} holds, we usually cannot expect that
such $\hat{c}$ exists. In the next section, we will observe that the
functions $\alpha_c$, $\beta_c$, and $\gamma$, used in the formula of
$\mathcal{A}_c$, should be taken carefully for such a
purpose.

%=============================================================================

\section{The proposed exact augmented Lagrangian function}
\label{sec:exact_auglag}

Here, we construct a particular exact augmented Lagrangian function by
choosing the functions $\alpha_c$, $\beta_c$, and $\gamma$, used
in~$\mathcal{A}_c$ (formula~\eqref{eq:A_c}) appropriately. Before
that, let us note that by defining
\begin{displaymath}
  \alpha_c(x,\Lambda) = \frac{c}{2}, \qquad 
  \beta_c(x,\Lambda) = \frac{2}{c}, \qquad
  \gamma(x,\Lambda) = 0,
\end{displaymath}
where $c > 0$ is the penalty parameter, we obtain the augmented
Lagrangian function for NSDP given in~\cite{CR04,SS04}. This function
is actually an extension of the classical augmented Lagrangian
function for NLP (see~\cite{Ber82a} for instance), and it is equal to
the Lagrangian function with some additional terms. However, it is not
exact in the sense of Definition~\ref{def:exactaug}.

In order to construct an augmented Lagrangian function with exactness
property, we choose a more complex~$\gamma$, that satisfies
Assumption~\ref{hip:general_func}(d),(e). As in~\cite{DPL96}, the
function~$\Gamma$ of item~(e) can be taken as a function that
estimates the value of the Lagrange multipliers associated to a
point. One possibility for such an estimate for NSDP problems is given
in~\cite{Han14}, which in turn extends the ones proposed
in~\cite{FSF12,GP79}. Basically, given $x \in \Re^n$, we consider the
following unconstrained problem:
\begin{equation}
  \label{eq:estimate}
  \begin{array}{ll}
    \underset{\Lambda}{\mbox{minimize}} & 
    \norm{\grad_x L(x,\Lambda)}^2 + \zeta_1^2 
    \norm{\L_{G(x)}(\Lambda)}^2_F + \zeta_2^2 r(x) \norm{\Lambda}^2_F\\
    \mbox{subject to} & \Lambda \in \S^m,
  \end{array}
\end{equation}
where $\zeta_1,\zeta_2 \in \Re$ are positive scalars, and $r \colon
\Re^n \to \Re$ denotes the residual function associated to the
feasible set, that is,
\begin{displaymath}
  r(x) := \frac{1}{2} \norm{\PS(-G(x))}^2_F = \frac{1}{2}
  \norm{\PS(G(x)) - G(x)}^2_F.
\end{displaymath}
Observe that $r(x) = 0$ if, and only if, $x$ is feasible
for~\eqref{eq:nsdp}. The idea underlying problem~\eqref{eq:estimate}
is to force KKT conditions~\eqref{eq:nsdp_kkt} to hold, except for the
feasibility of the Lagrange multiplier. Actually, this problem can be
seen as a linear least squares problem, and so its solution can be
written explicitly, when the so-called \emph{nondegeneracy} assumption
holds. It is well-known that the nondegeneracy condition,
defined below, extends the classical linear independence constraint
qualification for nonlinear programming~\cite{BS00,Sha97}, see 
also Section 4 and Corollary 2 in \cite{LFF16}. In particular, under nondegeneracy, 
Lagrange multiplies are ensured to exist at optimal points.

\begin{assumption}
  \label{hip:nondegeneracy}
  Every $x \in \Re^n$ feasible for~\eqref{eq:nsdp} is
  \emph{nondegenerate}, that is,
  \begin{displaymath}
    \S^m = \lin \mathcal{T}_{\S^m_+}(G(x)) + \mathrm{Im} \, \grad G(x),
  \end{displaymath}
  where $\mathcal{T}_{\S^m_+}(G(x))$ denotes the tangent cone of $\S^m_+$ at
  $G(x)$, $\mathrm{Im} \,\grad G(x)$ is the image of the linear map
  $\grad G(x)$, and $\lin$ means lineality space.
\end{assumption}

\begin{lemma}
  \label{lem:estimate}
  Suppose that Assumption~\ref{hip:nondegeneracy} holds. For a given
  $x \in \Re^n$, define $N \colon \Re^n \to \S^m$~as
  \begin{equation}
    \label{eq:N(x)}
    N(x) := \grad G(x) \grad G(x)^* + \zeta_1^2 \L_{G(x)}^2 + 
    \zeta_2^2 \, r(x) I.
  \end{equation}
  Then, the following statements are true.
  \begin{itemize}
  \item[(a)] $N(\cdot)$ is continuously differentiable and for all $x
    \in \Re^n$, the matrix $N(x)$ is positive definite.
  \item[(b)] The solution of problem~\eqref{eq:estimate} is unique and
    it is given by
    \begin{equation}
      \label{eq:Lambda}
      \Lambda(x) = N(x)^{-1} \grad G(x) \grad f(x).
    \end{equation}
  \item[(c)] If $(x,\Lambda) \in \Re^n \times \S^m$ is a KKT pair
    of~\eqref{eq:nsdp}, then $\Lambda(x) = \Lambda$.
  \item[(d)] The operator $\Lambda(\cdot)$ is continuously
    differentiable, and $\grad \Lambda(x) = N(x)^{-1} Q(x)$, where
    \begin{eqnarray*}
      Q(x) & := & \grad^2 G(x) \grad_x L(x,\Lambda(x)) + \grad G(x)
      \grad_{xx}^2 L(x,\Lambda(x)) \\
      & & {}- \zeta_1^2 \grad_x \big[ \L^2_{G(x)}
      (\Lambda(x)) \big] - \zeta_2^2 \grad r(x) \Lambda(x).
    \end{eqnarray*}
  \end{itemize}
\end{lemma}

\begin{proof}
  See~\cite[Lemma 2.2 and Proposition~2.3]{Han14}.
\end{proof}

The augmented Lagrangian function $L_c \colon \Re^n \times \S^m \to
\Re$ that we propose is given by
\begin{equation}
  \label{eq:exactaugL}
  L_c(x,\Lambda) := f(x) + \frac{1}{2c} \left( 
    \norm{\PS (\Lambda - c\,G(x))}^2_F - \norm{\Lambda}^2_F \right)
  + \norm{N(x)(\Lambda(x)-\Lambda)}^2_F,
\end{equation}
where $\Lambda(\cdot)$ and $N(\cdot)$ are given in
Lemma~\ref{lem:estimate}.  It is equivalent to the usual augmented
Lagrangian function for NSDP, except for the last term. So, comparing
to the generalized one~\eqref{eq:A_c}, we have
\begin{equation}
  \label{eq:A_c_functions}
  \alpha_c(x,\Lambda) = \frac{c}{2}, \qquad 
  \beta_c(x,\Lambda) = \frac{2}{c}, \qquad
  \gamma(x,\Lambda) = \norm{N(x)(\Lambda(x)-\Lambda)}^2_F.
\end{equation}
Observe that the functions $\alpha_c,\beta_c,\gamma$ defined in such a
way satisfy Assumption~\ref{hip:general_func}. In fact, items~(a),
(b), and~(c) of this assumption hold trivially, and item (d) is
satisfied because of Lemma~\ref{lem:estimate}(c). The
function~$\Gamma$ of item (e) corresponds to~$\Lambda(\cdot)$, with
the necessary properties described in Lemma~\ref{lem:estimate}(c),(d).
Note that $\gamma(x,\Lambda(x)) = 0$ for all~$x$ in this case.

Now, from~\eqref{eq:N(x)} and~\eqref{eq:Lambda}, observe that
\begin{eqnarray}
  N(x) (\Lambda(x) - \Lambda) & = & \grad G(x) \grad f(x) - \grad G(x) 
  \grad G(x)^* \Lambda - \zeta_1^2 \L_{G(x)}^2(\Lambda) 
  -\zeta_2^2 r(x) \Lambda \nonumber \\
  & = & \grad G(x) \grad_x L(x,\Lambda) - \zeta_1^2 \L_{G(x)}^2 (\Lambda)
  - \zeta_2^2 r(x) \Lambda.
  \label{eq:NLambda}
\end{eqnarray}
Also, consider the following auxiliary function $Y_c \colon \Re^n
\times \S^m \to \S^m$ defined by
\begin{equation}
  \label{eq:Y_c}
  Y_c(x,\Lambda) := \PS \left( \frac{\Lambda}{c} - G(x) \right)
  - \frac{\Lambda}{c}.
\end{equation}
The gradient of $L_c(x,\Lambda)$ with respect to $x$ is given by
\begin{eqnarray}
  \grad_x L_c(x,\Lambda) & = & \grad f(x) - \grad G(x)^* 
  \PS (\Lambda - c \, G(x)) \nonumber \\
  & & {} + 2 K(x,\Lambda)^* N(x)
  (\Lambda(x) - \Lambda) \label{eq:grad_x_Lc} \\
  & = & \grad_x L(x,\Lambda) - c \grad G(x)^* Y_c(x,\Lambda)
  + 2 K(x,\Lambda)^* N(x) (\Lambda(x) - \Lambda),\nonumber
\end{eqnarray}
with
\begin{eqnarray*}
  K(x,\Lambda) & := & \grad_x \big[ N(x)(\Lambda(x) - \Lambda) \big] \\
  & = & \grad_x \big[ \grad G(x) \grad_x L(x,\Lambda) - 
  \zeta_1^2 \L_{G(x)}^2 (\Lambda) - \zeta_2^2 r(x) \Lambda \big],
\end{eqnarray*}
where the second equality follows from~\eqref{eq:NLambda}. Using
Lemma~\ref{lem:derivatives}(a), as well as some additional
calculations, we obtain
\begin{displaymath}
  \begin{array}{@{}r@{\:\:}c@{\:\:}l}
    & & \hspace{-3pt} \grad_x L_c(x,\Lambda) = \\ 
    & & {} \grad_x L(x,\Lambda) - c 
    \grad G(x)^* Y_c(x,\Lambda) + 2 \grad_{xx}^2 L(x,\Lambda)
    \grad G(x)^* N(x) (\Lambda(x) - \Lambda) \\
    & & {} + \displaystyle{ 2 \left[ \left\langle 
          \frac{\partial^2 G(x)}{\partial x_i x_j},
          N(x) (\Lambda(x) - \Lambda) \right\rangle \right]_{i,j=1}^n
      \grad_x L(x,\Lambda)} \\
    & & {} - \displaystyle{ 2 \zeta_1^2 \left[ \left\langle 
          \frac{\partial G(x)}{\partial x_i} \circ 
          (G(x) \circ \Lambda) + G(x) \circ 
          \left( \frac{\partial G(x)}{\partial x_i} \circ \Lambda \right),
          N(x) (\Lambda(x) - \Lambda) \right\rangle \right]_{i=1}^n}\\
    & & {} - 2 \zeta_2^2 \inner{\Lambda}{N(x) (\Lambda(x) - 
      \Lambda)} \grad r(x).
  \end{array} 
\end{displaymath}
Moreover, the gradient of $L_c(x,\Lambda)$ with respect to $\Lambda$
can be written as follows:
\begin{equation}
  \label{eq:grad_L_Lc}
  \grad_\Lambda L_c(x,\Lambda) = Y_c(x,\Lambda) - 2 N(x)^2 
  (\Lambda(x) - \Lambda).
\end{equation}

Here, we point out that the formulas of $L_c$, $\grad_x L_c$ and
  $\grad_\Lambda L_c$, presented respectively in~\eqref{eq:exactaugL},
  \eqref{eq:grad_x_Lc} and~\eqref{eq:grad_L_Lc}, do not require
  explicit computation of the multiplier estimate $\Lambda(x)$. In
  fact, the estimate only appears in the expression $N(x) (\Lambda(x)
  - \Lambda)$, that can be written as~\eqref{eq:NLambda}. It means
  that both $L_c$ and their gradients do not require solving the
  linear least squares problem~\eqref{eq:estimate}, which is
  computational expensive.

%=============================================================================

\subsection{Exactness results}
\label{sec:exactness}

In the whole section, we suppose that Assumptions~\ref{hip:cq}
and~\ref{hip:nondegeneracy} hold. Indeed, it can be noted that the
assertion about constraint qualifications in Assumption~\ref{hip:cq}
holds automatically from Assumption~\ref{hip:nondegeneracy}. Here, we
will show that the particular augmented Lagrangian $L_c$, defined
in~\eqref{eq:exactaugL}, is in fact exact. With this purpose, we will
first establish the relation between the KKT points of the original
\eqref{eq:nsdp} problem and the stationary points of the unconstrained
problem:
\begin{equation}
  \label{eq:nlp}
  \begin{array}{ll}
    \underset{x,\Lambda}{\mbox{minimize}} & L_c(x,\Lambda) \\ 
    \mbox{subject to} & (x,\Lambda) \in \Re^n \times \S^m.
  \end{array}
\end{equation}

\begin{proposition}
  \label{prop:kkt_nsdp2nlp}
  Let $(x,\Lambda) \in \Re^n \times \S^m$ be a KKT pair
  of~\eqref{eq:nsdp}. Then, for all $c > 0$, $ L_c(x,\Lambda) = f(x)$
  and $(x,\Lambda)$ is stationary of~$L_c$, that is, $\grad_x
  L_c(x,\Lambda) = 0$ and $\grad_\Lambda L_c(x,\Lambda) = 0$.
\end{proposition}

\begin{proof}
  Recalling that the functions defined in~\eqref{eq:A_c_functions}
  satisfy Assumption~\ref{hip:general_func}, the result follows from
  Proposition~\ref{prop:nsdp2A_c}.
\end{proof}

\begin{proposition}
  \label{prop:kkt_nlp2nsdp}
  Let $\hat{x} \in \Re^n$ be feasible for~\eqref{eq:nsdp} and
  $\hat{\Lambda} \in \S^m$. So, there exist $\hat{c}, \hat{\delta}_1,
  \hat{\delta}_2 > 0$ such that if $(x,\Lambda) \in \Re^n \times \S^m$
  is stationary of $L_c$ with $\norm{x-\hat{x}} \le \hat{\delta}_1$,
  $\norm{\Lambda-\hat{\Lambda}}_F \le \hat{\delta}_2$ and $c \ge
  \hat{c}$, then $(x,\Lambda)$ is a KKT pair of \eqref{eq:nsdp}.
\end{proposition}

\begin{proof}
  Let us first consider an arbitrary pair $(x,\Lambda) \in \Re^n
  \times \S^m$ and $c > 0$. For convenience, we also define the
  following function:
  \begin{equation}
    \label{eq:tildeY_c}
    \hat{Y}_c(x,\Lambda) := Y_c(x,\Lambda) + G(x) = 
    \PS \left( -\frac{\Lambda}{c} + G(x) \right),
  \end{equation}
  with the last equality following from
  Lemma~\ref{lem:properties1}(a). Lemma~\ref{lem:properties1}(b) and
  the definition of~$Y_c(x,\Lambda)$ in~\eqref{eq:Y_c} show that
  \begin{displaymath}
    0 = \hat{Y}_c(x,\Lambda) \circ \PS \left( 
      \frac{\Lambda}{c} - G(x) \right) = \hat{Y}_c(x,\Lambda) \circ
    \left( Y_c(x,\Lambda) + \frac{\Lambda}{c} \right).
  \end{displaymath}
  The above expression can be rewritten using the distributivity of
  the Jordan product:
  \begin{displaymath}
    \hat{Y}_c(x,\Lambda) \circ Y_c(x,\Lambda) = - \hat{Y}_c(x,\Lambda)
    \circ \frac{\Lambda}{c}.
  \end{displaymath}
  Moreover, from~\eqref{eq:tildeY_c}, the above equality, the
  distributivity and the commutativity of the Jordan product, we
  obtain
  \begin{eqnarray*}
    G(x) \circ \Lambda
    & = & \hat{Y}_c(x,\Lambda) \circ \Lambda - Y_c(x,\Lambda) 
    \circ \Lambda \\
    & = & -c \hat{Y}_c(x,\Lambda) \circ Y_c(x,\Lambda)
    - Y_c(x,\Lambda) \circ \Lambda\\
    & = & - \big(c \hat{Y}_c(x,\Lambda) +  \Lambda \big) \circ Y_c(x,\Lambda).
  \end{eqnarray*}
  Using this expression, we have
  \begin{eqnarray}
    \frac{1}{c} \L_{G(x)}^2 (\Lambda) & = &
    - \frac{1}{c} G(x) \circ \Big[ \big( c 
    \hat{Y}_c(x,\Lambda) + \Lambda) \big) \circ Y_c(x,\Lambda) \Big] 
    \nonumber \\
    & = & - \L_{G(x)} \L_{\hat{Y}_c(x,\Lambda) + \Lambda /c} (Y_c(x,\Lambda)).
    \label{eq:prop:kkt_nlp2nsdp.1}
  \end{eqnarray}
  Now, the formula of $\grad_x L_c(x,\Lambda)$ in~\eqref{eq:grad_x_Lc} and 
  the equality \eqref{eq:NLambda} show that
  \begin{eqnarray*}
    \frac{1}{c} \grad G(x) \grad_x L_c(x,\Lambda) & = & 
    \frac{1}{c} \grad G(x) \grad_x L(x,\Lambda) - \grad G(x) 
    \grad G(x)^* Y_c(x,\Lambda)\\
    & & {} + \frac{2}{c} \grad G(x) K(x,\Lambda)^* N(x) 
    (\Lambda(x) - \Lambda)\\
    & = & \frac{1}{c} \big( I + 2 \grad G(x) K(x,\Lambda)^* \big) 
    N(x)(\Lambda(x) - \Lambda) \\ 
    & & {} + \frac{1}{c} \zeta_1^2 \L_{G(x)}^2 (\Lambda) + \frac{1}{c} 
    \zeta_2^2 r(x)\Lambda - \grad G(x) \grad G(x)^* Y_c(x,\Lambda).
  \end{eqnarray*}
  Thus, from~\eqref{eq:prop:kkt_nlp2nsdp.1}, we have
  \begin{eqnarray}
    \frac{1}{c} \grad G(x) \grad_x L_c(x,\Lambda) & = &
    \frac{1}{c} \big( I + 2 \grad G(x) K(x,\Lambda)^* \big) 
    N(x)(\Lambda(x) - \Lambda) \nonumber \\ 
    & & {} - N_c(x,\Lambda) Y_c(x,\Lambda) 
    + \frac{1}{c} \zeta_2^2 r(x)\Lambda, \label{eq:prop:kkt_nlp2nsdp.2}
  \end{eqnarray}
  where
  \begin{displaymath}
    N_c(x,\Lambda) := \grad G(x) \grad G(x)^* 
    + \zeta_1^2 \L_{G(x)} \L_{\hat{Y}_c(x,\Lambda) + \Lambda /c}.
  \end{displaymath}

  Let us now consider $(x,\Lambda) \in \Re^n \times \S^m$ that is
  stationary of $L_c$. Since $\grad_\Lambda L_{c}(x,\Lambda) = 0$, we
  obtain, from~\eqref{eq:grad_L_Lc},
  \begin{equation}
    \label{eq:prop:kkt_nlp2nsdp.3}
    \frac{1}{2} N(x)^{-1} Y_{c}(x,\Lambda) = \big( 
    N(x) (\Lambda(x) - \Lambda) \big),
  \end{equation}
  because $N(x)$ is nonsingular from Lemma~\ref{lem:estimate}(a).
  Recalling that $\grad_x L_{c}(x,\Lambda) = 0$ also holds,
  then, from \eqref{eq:prop:kkt_nlp2nsdp.2} we obtain
  \begin{equation}
    \label{eq:prop:kkt_nlp2nsdp.4}
    \tilde{N}_{c}(x,\Lambda) Y_{c}(x,\Lambda)
    - \frac{1}{c} \zeta_2^2 r(x) \Lambda = 0,
  \end{equation}
  where
  \begin{displaymath}
    \tilde{N}_{c}(x,\Lambda) := N_{c}(x,\Lambda) 
    - \frac{1}{2c} \big( I + 2 \grad G(x) K(x,\Lambda)^* \big)
    N(x)^{-1}.
  \end{displaymath}
  Using the fact that $\norm{W}_F^2/2 - \norm{Z}_F^2 \le
  \norm{W-Z}_F^2$ for any matrices $W,Z$,
  from~\eqref{eq:prop:kkt_nlp2nsdp.4}, we can write
  \begin{equation}
    \label{eq:prop:kkt_nlp2nsdp.5}
    \frac{1}{2} \norm{\tilde{N}_{c}(x,\Lambda)
    Y_{c}(x,\Lambda)}_F^2 - \frac{1}{c^2} \zeta_2^4 
    r(x)^2 \norm{\Lambda}_F^2 \le 0.
  \end{equation}
  Moreover, the definition of projection and
  Lemma~\ref{lem:properties1}(a) yield
  \begin{eqnarray*}
    r(x) & \le & \frac{1}{2} \left\| \PS \left(-\frac{\Lambda}{c} +
    G(x) \right) - G(x) \right\|_F^2 \\
    & = & \frac{1}{2} \left\| \PS \left(-\frac{\Lambda}{c} + G(x)
    \right) - \left( -\frac{\Lambda}{c} + G(x) \right) -
    \frac{\Lambda}{c} \right\|_F^2 \\
    & = & \frac{1}{2} \norm{Y_{c}(x,\Lambda)}_F^2.
  \end{eqnarray*}
  The above inequality, together with~\eqref{eq:prop:kkt_nlp2nsdp.5} implies
  \begin{equation}
    \label{eq:prop:kkt_nlp2nsdp.6}
    \left[ \frac{1}{2} \sigma_{\min}^2(\tilde{N}_{c}(x,\Lambda)) -
      \frac{1}{2c^2} \zeta_2^4 r(x) \norm{\Lambda}_F^2 \right]
    \norm{Y_{c}(x,\Lambda)}_F^2 \le 0,
  \end{equation}
  where $\sigma_{\min}(\cdot)$ denotes the smallest singular value
  function.

  Recalling that $\hat{x}$ is feasible for~\eqref{eq:nsdp}, we observe
  that if $c \to \infty$, then $\hat{Y}_c(\hat{x},\Lambda) \to
  \PS(G(\hat{x})) = G(\hat{x})$ for all $\Lambda$. Since $r(\hat{x}) =
  0$, this also shows that
  \begin{displaymath}
    c \to \infty \quad \implies \quad 
    N_c(\hat{x},\Lambda) \to N(\hat{x}) \quad \mbox{for all } 
    \Lambda \in \S^m,
  \end{displaymath}
  with $N(\hat{x})$ defined in~\eqref{eq:N(x)}. Now, note that from
  Lemma~\ref{lem:estimate}(a), $N(\hat{x})$ is positive definite.
  Also, define 
  \begin{displaymath}
    M(x) :=  \grad G(x) \grad G(x)^* + \zeta_1^2 \L_{G(x)} \L_{{\PS}(G(x))}.
  \end{displaymath}
  Observe that~$M$ is continuous because all functions involved in its
  formula are continuous, and that $M(\hat x) = N(\hat x)$, which is
  positive definite. Therefore, there is $\delta_1 > 0$ such that
  $\norm{x - \hat x} \leq \delta _1$ implies that $M(x)$ is also
  positive definite.

  Letting $\hat{\Lambda} \in \S^m$, there exist $c_0,\hat \delta _1,
  \hat{\delta}_2 > 0$ with $ \hat \delta _1 < \delta _1 $ such that
  both $M(x)$ and $N_{c_0} (x,\Lambda)$ are positive definite for all
  $(x,\Lambda)$ in the set
  \begin{displaymath}
    \mathcal{V} := \big\{ (x,\Lambda)
    \in \Re^n \times \S^m \colon \norm{x-\hat{x}} \le \hat{\delta}_1, 
    \norm{\Lambda-\hat{\Lambda}}_F \le \hat{\delta}_2 \big\}.
  \end{displaymath} 
  
  \noindent Now, we would like to prove that there is $\hat c_0 > 0$
  such that $N_{c}(x,\Lambda)$ is positive definite for all $c \geq
  \hat c_0$ and all $(x,\Lambda)$ in $\mathcal{V}$. To do so, suppose
  that this statement is false. Then, there are sequences $\{c_k\}
  \subset \Re_{++}$ and $\{(x^k,\Lambda_k)\} \in \mathcal{V}$ such
  that $c_k \to \infty$ and $N_{c_k}(x^k,\Lambda_k)$ is not positive
  definite for all $k$.  Since $\mathcal{V}$ is compact, we may assume
  that $\{(x^k,\Lambda_k)\}$ converges to some $(\tilde x, \tilde
  \Lambda) \in \mathcal{V}$.  However, we have
  \begin{displaymath}
    M(\tilde x) = \lim _{k \to \infty} N_{c_k}(x^k,\Lambda _k).
  \end{displaymath} 
  Since $M(\tilde x)$ is positive definite, $N_{c_k}(x^k,\Lambda _k)$
  should also be positive definite for~$k$ sufficiently large. This
  contradicts the fact that no $N_{c_k}(x^k,\Lambda _k)$ is positive
  definite, by construction. We conclude that there is $\hat c_0$ 
  such that $N_{c}(x,\Lambda)$ is
  positive definite for all $c \geq \hat c_0$ and all $(x,\Lambda)$ in
  $\mathcal{V}$.

  Considering one such $(x,\Lambda) \in \mathcal{V}$, we now seek some
  $c(x,\Lambda) \ge \hat c_0$ such that $\tilde{N}_c (x,\Lambda)$ is
  nonsingular for all $c \ge c(x,\Lambda)$. We remark that we already
  know that $\sigma _{\min}(N_c(x,\Lambda))$ is positive over
  $\mathcal{V}$. Denote by $\sigma _{\max}(\cdot)$ the maximum
  singular value function. Then, recalling the formula for
  $\tilde{N}_c (x,\Lambda)$ and elementary properties of singular
  values\footnote{Namely, that $\inf _{\norm{x} = 1} \norm{Ax - Bx}=
    \sigma _{\min }(A-B) \geq \sigma _{\min }(A) - \sigma _{\max
    }(B)$.}, we have that if
  \begin{displaymath}
  c(x,\Lambda) = \max\left(1 + \frac{\sigma _{\max}\big(\big( I + 2
      \grad G(x) K(x,\Lambda)^* \big)N(x)^{-1}\big)}{ 2 \sigma_{\min}
      (N_{c}(x,\Lambda))},\hat c_0 \right)
  \end{displaymath}
  then $\tilde{N}_c (x,\Lambda)$ will be nonsingular for all $c \geq
  c(x,\Lambda)$.  As $c(x,\Lambda)$ is a continuous function of
  $(x,\Lambda)$, we have that $c_1 := \sup_{\mathcal{V}} c(x,\Lambda)$
  is finite. Thus, $\tilde{N}_c (x,\Lambda)$ is positive definite, and
  hence $\sigma_{\min}(\tilde{N}_c (x,\Lambda)) > 0$ for all $c \ge
  c_1$ and $(x,\Lambda) \in \mathcal{V}$. Similarly, there exists
  $\hat{c} \ge c_1$ such that
  \begin{displaymath}
    \frac{1}{2} \sigma_{\min}(\tilde{N}_c (x,\Lambda)) -
    \frac{1}{2c^2} \zeta_2^4 r(x) \norm{\Lambda}_F^2 > 0
    \quad \mbox{for all } c \ge \hat{c},\quad (x,\Lambda) \in \mathcal{V}. 
  \end{displaymath}

  Finally, consider $(x,\Lambda) \in \Re^n \times \S^m$ and $c \in
  \Re_{++}$ such that $(x,\Lambda)$ is stationary of $L_c$,
  $\norm{x-\hat{x}} \le \hat{\delta}_1$,
  $\norm{\Lambda-\hat{\Lambda}}_F \le \hat{\delta}_2$ and $c \ge
  \hat{c}$. From~\eqref{eq:prop:kkt_nlp2nsdp.6} and the above
  inequality, it means that $Y_c (x,\Lambda) = 0$. Also, from
  Lemma~\ref{lem:properties2}, and the fact that $c > 0$, it yields
  \begin{displaymath}
    G(x) \in \S^m_+, \quad \Lambda \in \S^m_+ \quad \mbox{and} 
    \quad \Lambda \circ G(x) = 0. 
  \end{displaymath}
  Moreover, from~\eqref{eq:prop:kkt_nlp2nsdp.3} and the fact that
  $N(x)$ is nonsingular by Lemma~\ref{lem:estimate}(a), we have
  $\Lambda(x) = \Lambda$. Since $\grad_x L_c(x,\Lambda) = 0$ also
  holds, from~\eqref{eq:grad_x_Lc}, we obtain $\grad_x L(x,\Lambda) =
  0$. Therefore, $(x,\Lambda)$ is a KKT pair of~\eqref{eq:nsdp}.
\end{proof}

\begin{proposition}
  \label{prop:kkt_nlp2nsdp.2}
  Let $\{ x^k \} \subset \Re^n$, $\{ \Lambda_k \} \subset \S^m$, and
  $\{ c_k \} \subset \Re_{++}$ be sequences such that $c_k \to \infty$
  and $(x^k,\Lambda_k)$ is stationary of~$L_{c_k}$ for all $k$. Assume
  that there are subsequences $\{ x^{k_j} \}$ and $\{ \Lambda_{k_j}
  \}$ of $\{ x^k \}$ and $\{ \Lambda_k \}$, respectively, such that
  $x^{k_j} \to \hat{x}$ and $\Lambda_{k_j} \to \hat{\Lambda}$ for some
  $(\hat{x},\hat{\Lambda}) \in \Re^n \times \S^m$. Then, either there
  exists $\hat{k} > 0$ such that $(x^{k_j},\Lambda_{k_j})$ is a KKT
  pair of~\eqref{eq:nsdp} for all $k_j \ge \hat{k}$, or $\hat{x}$ is a
  stationary point of the residual function~$r$ that is infeasible
  for~\eqref{eq:nsdp}.
\end{proposition}

\begin{proof}
  We first show that $\hat{x}$ is a stationary point of~$r$, in other
  words,
  \begin{displaymath}
    \grad r(\hat{x}) = - \grad G(\hat{x})^* \PS(-G(\hat{x})) = 0.
  \end{displaymath}
  In fact, using~\eqref{eq:grad_x_Lc} and dividing the equation
  $\grad_x L_{c_{k_j}} (x^{k_j},\Lambda_{k_j}) = 0$ by $c_{k_j}$, we have
  \begin{eqnarray*}
    \frac{1}{c_{k_j}} \grad_x L(x^{k_j},\Lambda_{k_j}) - \grad G(x^{k_j})^* 
    Y_{c_{k_j}}(x^{k_j},\Lambda_{k_j}) + \\
    \frac{2}{c_{k_j}} K(x^{k_j},\Lambda_{k_j})^* N(x^{k_j}) 
    (\Lambda(x^{k_j}) - \Lambda_{k_j}) = 0.
  \end{eqnarray*}
  Recalling Lemma~\ref{lem:estimate}, we observe that all the
  functions involved in the above equation are continuous. Thus,
  taking the limit $k_j \to \infty$, from the definition of $Y_c$
  in~\eqref{eq:Y_c}, we obtain $-\grad G(\hat{x})^* \PS(-G(\hat{x})) =
  0$, as we claimed. Now, assume that $\hat{x}$ is feasible. Then,
  from Proposition~\ref{prop:kkt_nlp2nsdp}, there exists $\hat{k} > 0$
  such that $(x^{k_j},\Lambda_{k_j})$ is a KKT pair of~\eqref{eq:nsdp}
  for all $k_j \ge \hat{k}$, which completes the proof.
\end{proof}

Now, recalling Definition~\ref{def:exactaug}, once again, we use the
notations $G_{\mbox{\tiny{NSDP}}}$ ($L_{\mbox{\tiny{NSDP}}}$) and
$G_{\mbox{\tiny{NLP}}}(c)$ ($L_{\mbox{\tiny{NLP}}}(c)$) to denote the sets
of global (local) minimizers of problems~\eqref{eq:nsdp}
and~\eqref{eq:nlp}, respectively. The following theorems show that the
proposed function $L_c$ given in~\eqref{eq:exactaugL} is in fact an
exact augmented Lagrangian function. However, the results are
established as in~\cite{AFS13}, where it is admitted that we can end
up with a stationary point of the residual function~$r$ that is
infeasible for~\eqref{eq:nsdp}.

\begin{theorem}
  Let $\{ x^k \} \subset \Re^n$, $\{ \Lambda_k \} \subset \S^m$, and
  $\{ c_k \} \subset \Re_{++}$ be sequences such that $c_k \to \infty$
  and $(x^k,\Lambda_k) \in L_{\mbox{\tiny{NLP}}}(c_k)$ for all~$k$.
  Assume that there are subsequences $\{ x^{k_j} \}$ and $\{
  \Lambda_{k_j} \}$ of $\{ x^k \}$ and $\{ \Lambda_k \}$,
  respectively, such that $x^{k_j} \to \hat{x}$ and $\Lambda_{k_j} \to
  \hat{\Lambda}$ for some $(\hat{x},\hat{\Lambda}) \in \Re^n \times
  \S^m$. Then, either there exists $\hat{k} > 0$ such that $x^{k_j}
  \in L_{\mbox{\tiny{NSDP}}}$, with an associated Lagrange multiplier
  $\Lambda_{k_j}$ for all $k_j \ge \hat{k}$, or $\hat{x}$ is a
  stationary point of the residual function~$r$ that is infeasible
  for~\eqref{eq:nsdp}.
\end{theorem}

\begin{proof}
  From Proposition~\ref{prop:kkt_nlp2nsdp.2}, either there exists
  $\hat{k} > 0$ such that $(x^{k_j},\Lambda_{k_j})$ is a KKT pair
  of~\eqref{eq:nsdp} for all $k_j \ge \hat{k}$, or $\hat{x}$ is a
  stationary point of~$r$ that is infeasible for~\eqref{eq:nsdp}. So,
  the result follows in the same way as the proof of
  Theorem~\ref{the:global_local}(b).
\end{proof}

For the above result, that concerns local minimizers, we note that the
existence of the subsequence $\{ x^{k_j} \}$ is guaranteed, for
example, when the whole sequence $\{ x^k \}$ is bounded. Moreover, if
the constraint function~$G$ is convex with respect to the cone
$\S^m_+$, then the residual function~$r$ is also convex, which means
that all stationary points of~$r$ are feasible for~\eqref{eq:nsdp}. In
the case of global minimizers, it is possible to prove full
equivalence, and we do not have to concern about stationary points
of~$r$ that are infeasible for~\eqref{eq:nsdp}.

\begin{proposition}
  \label{prop:global_L_c}
  Let $\{ x^k \} \subset \Re^n$, $\{ \Lambda_k \} \subset \S^m$, and
  $\{ c_k \} \subset \Re_{++}$ be sequences such that $x^k \to
  \hat{x}$ and $\Lambda_k \to \hat{\Lambda}$ for some
  $(\hat{x},\hat{\Lambda}) \in \Re^n \times \S^m$, $c_k \to \infty$,
  and $(x^k,\Lambda_k) \in G_{\mbox{\tiny{NLP}}}(c_k)$ for
  all~$k$. Then, there exists $\hat{k} > 0$ such that $x^k \in
  G_{\mbox{\tiny{NSDP}}}$ with an associated Lagrange multiplier
  $\Lambda_k$ for all $k \ge \hat{k}$.
\end{proposition}

\begin{proof}
  Let $\bar{x} \in G_{\mbox{\tiny{NSDP}}}$, which exists by
  Assumption~\ref{hip:cq}. Because of the nondegeneracy constraint
  qualification, there exists $\bar{\Lambda} \in \S^m$ such that
  $(\bar{x},\bar{\Lambda})$ is a KKT pair of~\eqref{eq:nsdp}. So, from
  Proposition~\ref{prop:kkt_nsdp2nlp}, $f(\bar{x}) =
  L_c(\bar{x},\bar{\Lambda})$ holds for all $c > 0$. Moreover, since
  $(x^k,\Lambda_k) \in G_{\mbox{\tiny{NLP}}}(c_k)$, we have
  \begin{equation}  
    \label{eq:prop:global_L_c.1}
    L_{c_k}(x^k,\Lambda_k) \le L_{c_k}(\bar{x},\bar{\Lambda}) =
    f(\bar{x})
  \end{equation}
  for all~$k$. Taking the supremum limit in this inequality, we
  obtain
  \begin{equation}
    \label{eq:limsup}
    \limsup_{k \to \infty} L_{c_k}(x^k,\Lambda_k) \le f(\bar{x}).
  \end{equation}
  Observe now that the formula of $L_{c_k}(x^k,\Lambda_k)$
  in~\eqref{eq:exactaugL} can be written equivalently as
  \begin{eqnarray}
    L_{c_k}(x^k,\Lambda_k) & = & f(x^k) + \frac{c_k}{2} \left( \left\| 
        \PS \left( \frac{\Lambda_k}{c_k} - G(x^k) \right) \right\|^2_F 
      - \left\| \frac{\Lambda_k}{c_k} \right\|^2_F \right) \nonumber \\
    & & {} + \norm{N(x^k)(\Lambda(x^k)-\Lambda_k)}^2_F \label{eq:againL_c}.
  \end{eqnarray}
  Recall from Lemma~\ref{lem:estimate} that the functions involved in
  the above equality are all continuous. This fact, together with
  inequality~\eqref{eq:limsup}, shows that $\PS(-G(\hat{x})) = 0$,
  that is, $\hat{x}$ is feasible. So, from
  Proposition~\ref{prop:kkt_nlp2nsdp.2}, we conclude that there exists
  $\hat{k} > 0$ such that $(x^k,\Lambda_k)$ is a KKT pair
  of~\eqref{eq:nsdp} for all $k \ge \hat{k}$. Since $c_k > 0$ and the
  norm is always nonnegative, \eqref{eq:againL_c} implies
  $L_{c_k}(x^k,\Lambda_k) \ge f(x^k) -
  \norm{\Lambda_k}_F^2/(2c_k)$. Again, taking the supremum limit in
  such an inequality, we have
  \begin{displaymath}
    \limsup_{k \to \infty} L_{c_k}(x^k,\Lambda_k) \ge f(\hat{x}), 
  \end{displaymath}
  which, together with~\eqref{eq:limsup} shows that $f(\hat{x}) \le
  f(\bar{x})$. Thus, $\hat{x} \in G_{\mbox{\tiny{NSDP}}}$ holds.

  Now, since $\hat{x}$ is feasible for~\eqref{eq:nsdp}, there exist
  $\hat{c},\hat{\delta}_1,\hat{\delta}_2$ as in
  Proposition~\ref{prop:kkt_nlp2nsdp}.  Consider~$\hat{k}$ large
  enough so that $\norm{x^k - \hat{x}} \le \hat{\delta}_1$,
  $\norm{\Lambda_k -\hat{\Lambda}}_F \le \hat{\delta}_2$, $c_k \ge
  \hat{c}$, and $(x^k,\Lambda_k) \in G_{\mbox{\tiny{NLP}}}(c_k)$ for
  all $k \ge \hat{k}$. Since $(x^k,\Lambda_k)$ is stationary of
  $L_{c_k}$, from Proposition~\ref{prop:kkt_nlp2nsdp}, we obtain that
  $(x^k,\Lambda_k)$ is also a KKT pair of~\eqref{eq:nsdp} for all $k
  \ge \hat{k}$. Once again from Proposition~\ref{prop:kkt_nsdp2nlp}
  and~\eqref{eq:prop:global_L_c.1}, we have $f(x^k) =
  L_{c_k}(x^k,\Lambda_k) \le f(\bar{x})$. Therefore, $x^k \in
  G_{\mbox{\tiny{NSDP}}}$ for all $k \ge \hat{k}$.
\end{proof}

\begin{theorem}
 \label{the:global_final}
  Assume that there exists $\bar{c} > 0$ such that $\bigcup_{c \ge
    \bar{c}} G_{\mbox{\tiny{NLP}}}(c)$ is bound\-ed. Then, there exists
  $\hat{c} > 0$ such that $G_{\mbox{\tiny{NLP}}}(c) =
  \tilde{G}_{\mbox{\tiny{NSDP}}}$ for all $c \ge \hat{c}$, where
  $\tilde{G}_{\mbox{\tiny{NSDP}}}$ is defined
  in~\eqref{eq:tilde_G_NSDP}.
\end{theorem}

\begin{proof}
  From Lemma~\ref{lem:global_subset}, we only need to show the
  existence of $\hat{c} > 0$ such that $G_{\mbox{\tiny{NLP}}}(c)
  \subseteq \tilde{G}_{\mbox{\tiny{NSDP}}}$ for all $c \ge
  \hat{c}$. Assume that this statement is false. Then, there exist
  sequences $\{ (x^k,\Lambda_k) \} \subset \Re^n \times \S^m$ and $\{
  c_k \} \subset \Re_{++}$ with $c_k \to \infty$, $c_k \ge \bar{c}$,
  and $(x^k,\Lambda_k) \in G_{\mbox{\tiny{NLP}}}(c_k)$, but such that
  $(x^k,\Lambda_k) \notin \tilde{G}_{\mbox{\tiny{NSDP}}}$. Since
  $\bigcup_{c \ge \bar{c}} G_{\mbox{\tiny{NLP}}}(c)$ is bounded, we
  can assume, without loss of generality, that $x^k \to \hat{x}$ and
  $\Lambda_k \to \hat{\Lambda}$ for some $(\hat{x},\hat{\Lambda}) \in
  \Re^n \times \S^m$. Thus, Proposition~\ref{prop:global_L_c} shows
  that there exists $\hat{k} > 0$ such that $(x^k,\Lambda_k) \in
  \tilde{G}_{\mbox{\tiny{NSDP}}}$ for all $k \ge \hat{k}$, which is a
  contradiction.
\end{proof}

%=============================================================================

\section{Preliminary numerical experiments}
\label{sec:experiments}

This work is focused on the theoretical aspects of exact augmented
Lagrangian functions, however, we take a look at the numerical
prospects of our approach by examining two simple problems in the next
two subsections. We will now explain briefly our proposal. Given a
problem~\eqref{eq:nsdp}, the idea is to use some unconstrained
optimization method to solve~\eqref{eq:nlp}, i.e., minimization
of~$L_c$ given in~\eqref{eq:exactaugL}. First, an initial point~$x^0
\in \Re^n$, together with the initial penalty parameter $c_0 > 0$ are
selected. Then, we choose the initial Lagrange multiplier~$\Lambda_0
\in \S^m$. One possibility is to use the multiplier estimate, i.e., to
set~$\Lambda_0$ as $\Lambda(x^0)$, where $\Lambda(\cdot)$ is defined
in~\eqref{eq:Lambda}.

Then, we run the unconstrained optimization method of our choice.
However, since the penalty values for which~$L_c$ becomes exact is not
known beforehand, we attempt to adjust the penalty parameter between
iterations as follows. Let $\tau \in (0,1)$ and $\rho > 1$. Denote
by~$x^k$ and~$\Lambda_k$, $c_k$ the values of~$x$, $\Lambda$ and~$c$
at the $k$th iteration, respectively. Recalling the function
$Y_{c}(\cdot,\cdot)$ defined in \eqref{eq:Y_c}, if at the $k$th
iteration we have
$$
\norm{Y_{c_k}(x^k,\Lambda_k)}_F = \left\| \PS \left(
  \frac{\Lambda_k}{c_k} - G(x^k) \right) - \frac{\Lambda_k}{c_k}
\right\|_F > \tau \norm{Y_{c_{k-1}}(x^{k-1},\Lambda_{k-1})}_F,
$$
then we let $c_k$ be $\rho \, c_{k-1}$. Otherwise, we let $c_k$ be
$c_{k-1}$. This is an idea that appears in many augmented Lagrangian
methods, for example in~\cite{BM14}. The motivation is that
$\norm{Y_{c_k}(x^k,\Lambda_k)}_F$ is zero if and only if $G(x^k) \in
\S^m_+$, $\Lambda_k \in \S^m_+$ and $G(x^k) \circ \Lambda _k =
0$. Therefore, $\norm{Y_{c_k}(x^k,\Lambda_k)}_F$ is a measure of the
degree to which complementarity and feasibility are satisfied, taking
into account the current penalty parameter.  In summary, whenever
there is not enough progress, we increase the penalty. In order to
avoid the problem becoming too ill-conditioned, we never increase the
penalty past some fixed value $c_{\max}$. We also point out that the
update of the penalty parameter can be done by using the so-called
test function, which is originally defined in~\cite{GP79}. However, as
it can be seen in the paper about exact penalty
functions~\cite{AFS13}, the above approach using $Y_c(\cdot,\cdot)$ is
more efficient, which justifies its use here.

In our implementation, $\rho$, $\tau$ and $c_{\max}$ is set to $1.1$,
$0.9$ and~$1000$, respectively. The maximum number of iterations
is~$5000$. The values of $\zeta_1$ and $\zeta_2$, which control the
behavior of $N(x)$ in \eqref{eq:N(x)}, are set to $1$ and $10^{-4}$,
respectively. The initial penalty parameter $c_0$ is computed using
the formula:
\begin{equation*}
  c_0 := \max \left\{ c_{\min}, \min \left\{ c_{\max}, 
      \frac{10 \max (1, |f(x^0)|)}{\max(1,0.5\norm{G(x^0)}_F^2)}
    \right\} \right\},
\end{equation*}
with $c_{\min} = 0.1$, which is similar to the one used
in~\cite{BM14}. The unconstrained method of our choice is the BFGS
method using the Armijo's condition for the line search.  We stop the
algorithm when the KKT conditions are satisfied within $10^{-5}$ or
when the norm of the gradient of $L_c$ is less than $10^{-5}$.

We implemented the algorithm in Python and ran all the experiments on
a Intel Core i7-6700 machine with 8 cores and 16GB of memory. As we
already pointed out after \eqref{eq:grad_L_Lc}, an important
implementational aspect is that we never need to explicitly evaluate
the function $\Lambda (\cdot)$, except in the optional way of
computing the initial Lagrange multiplier.

\subsection{Noll's example}

As an initial example, we took a look at this simple instance by Noll~\cite{N07}:
\begin{equation}
  \label{eq:noll}
  \tag{Noll}
  \begin{array}{ll}
    \underset{x \in \Re^2}{\mbox{minimize}} & 0.5 (-x_1^2-x_2^2) \\ 
    \mbox{subject to} &
    \left(
      \begin{array}{c@{\:\:\:}c@{\:\:\:}c}
        1 & x_1-1 & 0\\
        x_1-1 & 1 & x_2 \\
        0 & x_2 & 1
      \end{array} 
    \right)
    \in \S_{+}^3.
  \end{array}
\end{equation}
The problem \eqref{eq:noll} is already in the format 
\eqref{eq:nsdp}, which can be seen by letting $G$ be the function defined by
$$
G(x_1,x_2) := 
\left(
  \begin{array}{c@{\:\:\:}c@{\:\:\:}c}
    1 & x_1-1 & 0\\
    x_1-1 & 1 & x_2 \\
    0 & x_2 & 1
  \end{array} 
\right).
$$
In order to compute $L_c$ and its gradient (see \eqref{eq:exactaugL},
\eqref{eq:NLambda}, \eqref{eq:grad_x_Lc} and \eqref{eq:grad_L_Lc}), we
need the partial derivative matrices of~$G$ and the adjoint of the
gradient of~$G$, which are given below:
\begin{displaymath}
  \frac{\partial G(x)}{\partial x_{1}}
  = \left(
    \begin{array}{c@{\:\:}c@{\:\:}c}
      0 & 1 & 0\\
      1 & 0 & 0 \\
      0 & 0 & 0
    \end{array}
  \right), \quad
  \frac{\partial G(x)}{\partial x_{2}}
  = \left(
    \begin{array}{c@{\:\:}c@{\:\:}c}
      0 & 0 & 0\\
      0 & 0 & 1 \\
      0 & 1 & 0
    \end{array}
  \right), \quad
  \grad G(x)^*V = 2(V_{12},V_{23})^\T, 
\end{displaymath}
where $V \in \S^3$ is an arbitrary matrix with $(i,j)$ entry denoted
by~$V_{ij}$.

The optimal value of \eqref{eq:noll} is $-2$ and it is achieved at
$(2,0)$.  Starting at $(1,0)$, our method found a solution satisfying
the optimality criteria in~14 iterations and $0.01$ seconds. The
initial and final penalty parameters were $6.66$ and $10.74$,
respectively. The objective function, the constraint function and
their gradients were evaluated $41$ times each.

\subsection{The closest correlation matrix problem} 

Let $H$ be a $m \times m$ symmetric matrix. 
The goal is to find a correlation matrix~$X$ that is as close as possible to $H$. 
In other words, we seek a solution to the following problem:
\begin{equation}
\label{eq:cor1}
\tag{Cor}
\begin{array}{ll}
\underset{X}{\mbox{minimize}} & \inner{X - H}{X-H} \\ 
\mbox{subject to} & X_{ii} = 1 \quad \mbox{for all } i,\\
& X  \in \S_{+}^m.
\end{array}
\end{equation}

\noindent There are many variants of \eqref{eq:cor1} where weighting
factors are added, constraints on the eigenvalues are considered, and
so on. With that, this family of problems has found of wealth of
applications in statistics and finance~\cite{Hig02}.

In this example, it is possible to show that the nondegeneracy
condition is satisfied at every feasible point (Assumption
\ref{hip:nondegeneracy}), which guarantee the theoretical properties
of the exact augmented Lagrangian function~$L_c$. This was proved by
Qi and Sun in \cite{QS11}, but since there are some differences in
notation, we will first take a look at this issue.  Qi and Sun proved
the following result.

\begin{proposition}\label{prop:QS}
	Let $Y \in \S^m_+$ be such that $Y_{ii} = 1$, for all $i$. 
	Then, 
	$$
	\diag(\lin \mathcal{T}_{\S^m_+}(Y)) = \Re^m,
	$$	
	where $\diag \colon \S^m \to \Re^m$ is the linear map that maps a 
	symmetric matriz $Z$ to its diagonal $(Z_{11},Z_{22},\ldots, Z_{mm})$.
\end{proposition}
\begin{proof}
	See Proposition {2.1} and Equation (2.2) in \cite{QS11}.\qed
\end{proof}

We now write \eqref{eq:cor1} in a format similar to 
\eqref{eq:nsdp}. For that, denote by $A^{ij}$ the $m\times m$ 
the matrix that has $1$ in the $(i,j)$ and $(j,i)$ entries and $0$ elsewhere. 
Then, by discarding constant terms in the objective function,  \eqref{eq:cor1} can be reformulated  equivalently as follows.
\begin{align}
\underset{x \in \Re^{m(m-1)/2}}{\mbox{minimize}} & \quad \sum _{1 \leq i < j \leq m} (H_{ij}-x_{ij})^2 \tag{Cor2} \label{eq:cor_std}\\ 
\mbox{subject to} &\quad  I + \sum _{1 \leq i < j\leq m} A^{ij}x_{ij} \in \S^m_+. \notag
\end{align}
Here, $x$ can be thought as an upper triangular matrix without the diagonal. That is why the dimension of $x$ is $m(m-1)/2$  and we index $x$ by using $x_{ij}$ for $1 \leq i < j \leq m$.

\begin{proposition}
  Problem~\eqref{eq:cor_std} satisfies Assumption \ref{hip:nondegeneracy}.
\end{proposition}
\begin{proof}
	We must show that 
	\begin{displaymath}
	\S^m = \lin \mathcal{T}_{\S^m_+}(G(x)) + \mathrm{Im} \, \grad G(x),
	\end{displaymath}
	where $G: \Re^{m(m-1)/2} \to \S^m$ is the function such that 
	$$
	G(x) = I + \sum _{1\leq i < j \leq m} A^{ij}x_{ij}.
	$$ 	
	Now, let $Y \in \S^m$ be arbitrary. We can write 
	$Y$ as 
	$$
	Y = \sum _{i=1}^m A^{ii} Y_{ii} + \sum _{1 \leq i < j \leq m} A^{ij} Y_{ij}.
	$$
	From Proposition~\ref{prop:QS}, the first summation belongs to 
	$\lin \mathcal{T}_{\S^m_+}(Y)$. Then, noting that  $\mathrm{Im} \, \grad G(x) $ is the space spanned 
	by $\{A^{ij} \colon 1 \leq i < j \leq m \}$, we conclude 
	that the second summation belongs 
	to $\mathrm{Im} \, \grad G(x)$. 
	This shows that Assumption~\ref{hip:nondegeneracy} is satisfied. \qed
\end{proof}
We now write some useful formulae which can be 
used in conjunction with \eqref{eq:NLambda}, \eqref{eq:grad_x_Lc} and \eqref{eq:grad_L_Lc} to compute $L_c$ and its gradient.
Let $V$ be an arbitrary $m\times m$ symmetric matrix. Then,
\begin{displaymath}
\frac{\partial G(x)}{\partial x_{ij}} = A^{ij}, \qquad \grad G(x)^*V = 2v, 
\end{displaymath}
for $1 \leq i < j \leq m$, where~$v$ corresponds to the upper
triangular part of $V$ without the diagonal.

We now move on to the experiments. 
We generated $50$ symmetric matrices~$H$ such that the diagonal entries are all~$1$ and non-diagonal elements are uniform random numbers between $-1$ and $1$. This was repeated for $m = 5, 10, 15, 20$. We then ran our algorithm using as initial point the matrix having~$1$ in all its entries.
The results can be seen in Table \ref{table:ex1}.
All the values depicted in Table \ref{table:ex1} are averages among $50$ runs. The column ``Iterations'' correspond to average number of BFGS iterations.
At each run, we recorded the number of function 
evaluations for $f$, which is the same for $G,\grad f$ and $\grad G$. 
Then, the column ``Evaluations'' in Table \ref{table:ex1} is the average number of function evaluations.
Columns ``Initial $c$'' and ``Final $c$'' correspond to the average of the initial and final penalty parameters, respectively.
Finally, column ``Time~(s)'' is the average running time, in seconds.
\begin{table}[!htb]
	\centering
	\caption{Results for \eqref{eq:cor1}.}
	\label{table:ex1}
	\begin{tabular}{@{}rrrrrr}
		$m$ & Iterations & Evaluations  & Initial $c$ & Final $c$ & Time (s)  \\\midrule
		5	& 114.62		& 371.22			& 21.98 & 805.2   & 0.208\\
		10	& 520.96		& 1844.62			& 23.31 & 1000.0  & 1.923\\
		15	& 1191.62		& 4297.74			& 24.77 & 1000.0  & 10.170\\
		20	& 2101.02		& 7801.00			& 25.39 & 1000.0  & 42.490\\
	\end{tabular}
\end{table}

No failures were detected, that is, we obtained approximate KKT points within $10^{-5}$ for all the instances.
We also observed that, except for $m = 5$, the final penalty parameter climbed up to the maximum value.  
At first glance, this suggests that the penalty was not large enough. However, we were still able to solve the problems without increasing the maximum value.
In fact, we noted that, in some cases, the performance degraded when the maximum penalty value was increased.
At this moment, the method presented here is not competitive against the approach 
in \cite{QS11}, where we observed that an $20\times 20$ instance is typically solved in less than a second in our hardware. However, it should be emphasized that a second-order method is used in \cite{QS11}, where here we used 
BFGS. It would be interesting to apply and analyze a second-order method in combination with the exact augmented Lagrangian function $L_c$, but this investigation is beyond the scope of this paper.

%=============================================================================

\section{Final remarks}
\label{sec:conclusions}

We proposed a generalized augmented Lagrangian function
$\mathcal{A}_c$ for NSDP problems, giving conditions for it to be
exact. After that, we considered a particular function $L_c$, and we
proved that it is exact under the nondegeneracy condition and some
reasonable assumptions as in Theorem~\ref{the:global_final}. We
  also presented some preliminary numerical experiments using a
  quasi-Newton method with BFGS formula, showing the validity of
  the approach. One future work is to analyze more efficient
methods that can solve the unconstrained minimization of~$L_c$. From
Lemma~\ref{lem:estimate} and the formula of~$L_c$, given
in~\eqref{eq:exactaugL}, we observe that $L_c$ is an $\mbox{SC}^1$
function, i.e., it is continuously differentiable and its gradient is
semismooth. It means that the unconstrained problem can be solved with
second-order methods, as the semismooth Newton. However, the gradient
of $L_c$, given in~\eqref{eq:grad_x_Lc}, contains second-order terms
of problem functions~$f$ and~$G$, and thus, a second-order method
would have to deal with third-order derivatives. We believe that we
can use the idea proposed in~\cite{FSF12}, that avoids these
third-order terms, but still guaranteeing the global superlinear
convergence.

By using the generalized function~$\mathcal{A}_c$ as a tool, other
practical exact augmented Lagrangian functions can be studied for
NSDP, or other important conic optimization problems. In fact, as it
can be seen in~\cite{DPL96}, many other exact augmented Lagrangian
functions exist, but only for the classical nonlinear programming. For
example, recalling~\eqref{eq:A_c_functions}, we note that $L_c$ is
defined by choosing functions $\alpha_c$ and $\beta_c$ as
constants. By considering more sophisticated formulas, it is possible
to weaken the assumptions used here. This should be another matter of
investigation.

%=============================================================================

\section*{Acknowledgements}
  We would like to thank the anonymous referees for their suggestions
  which improved the original version of the paper. We are also
  thankful to Akiko Kobayashi for valuable discussions about exact
  augmented Lagrangian functions.

%=============================================================================

\bibliographystyle{plain}
\bibliography{journal-titles,references}

\begin{thebibliography}{10}

\bibitem{AS10}
T.~A. Andr{\' e} and P.~J.~S. Silva.
\newblock Exact penalties for variational inequalities with applications to
  nonlinear complementarity problems.
\newblock {\em Computational Optimization and Applications}, 47(3):401--429,
  2010.

\bibitem{AFS13}
R.~Andreani, E.~H. Fukuda, and P.~J.~S. Silva.
\newblock A {G}auss-{N}ewton approach for solving constrained optimization
  problems using differentiable exact penalties.
\newblock {\em Journal of Optimization Theory and Applications},
  156(2):417--449, 2013.

\bibitem{ANT03}
P.~Apkarian, D.~Noll, and H.~D. Tuan.
\newblock Fixed-order ${H}_\infty$ control design via a partially augmented
  {L}agrangian method.
\newblock {\em International Journal of Robust and Nonlinear Control},
  13(12):1137--1148, 2003.

\bibitem{BJKNZ00}
A.~Ben-Tal, F.~Jarre, M.~Ko{\v c}vara, A.~Nemirovski, and J.~Zowe.
\newblock Optimal design of trusses under a nonconvex global buckling
  constraint.
\newblock {\em Optimization and Engineering}, 1(2):189--213, 2000.

\bibitem{Ber09}
D.~S. Bernstein.
\newblock {\em Matrix Mathematics: Theory, Facts, and Formulas}.
\newblock Princeton University Press, 2nd edition, 2009.

\bibitem{Ber82a}
D.~P. Bertsekas.
\newblock {\em Constrained Optimization and {L}agrange Multipliers Methods}.
\newblock Academic Press, New York, 1982.

\bibitem{BM14}
E.~G. Birgin and J.~M. Mart\'inez.
\newblock {\em Practical augmented {L}agrangian methods for constrained
  optimization}.
\newblock Society for Industrial and Applied Mathematics, Philadelphia, PA,
  2014.

\bibitem{BS00}
J.~F. Bonnans and A.~Shapiro.
\newblock {\em Perturbation Analysis of Optimization Problems}.
\newblock Springer-Verlag, New York, 2000.

\bibitem{CR04}
R.~Correa and H.~Ram\'irez~C.
\newblock A global algorithm for nonlinear semidefinite programming.
\newblock {\em SIAM Journal on Optimization}, 15(1):303--318, 2004.

\bibitem{DPG79}
G.~Di~Pillo and L.~Grippo.
\newblock A new class of augmented {L}agrangians in nonlinear programming.
\newblock {\em SIAM Journal on Control and Optimization}, 17(5):618--628, 1979.

\bibitem{DPG82}
G.~Di~Pillo and L.~Grippo.
\newblock A new augmented {L}agrangian function for inequality constraints in
  nonlinear programming.
\newblock {\em Journal of Optimization Theory and Applications},
  36(4):495--519, 1982.

\bibitem{DPG89}
G.~Di~Pillo and L.~Grippo.
\newblock Exact penalty functions in constrained optimization.
\newblock {\em SIAM Journal on Control and Optimization}, 27(6):1333--1360,
  1989.

\bibitem{DPLLP03}
G.~Di~Pillo, G.~Liuzzi, S.~Lucidi, and L.~Palagi.
\newblock An exact augmented {L}agrangian function for nonlinear programming
  with two-sided constraints.
\newblock {\em Computational Optimization and Applications}, 25:57--83, 2003.

\bibitem{DPL96}
G.~Di~Pillo and S.~Lucidi.
\newblock On exact augmented {L}agrangian functions in nonlinear programming.
\newblock In G.~Di~Pillo and F.~Giannessi, editors, {\em Nonlinear Optimization
  and Applications}, pages 85--100. Springer, Boston, MA, 1996.

\bibitem{DPL01}
G.~Di~Pillo and S.~Lucidi.
\newblock An augmented {L}agrangian function with improved exactness
  properties.
\newblock {\em SIAM Journal on Optimization}, 12(2):376--406, 2001.

\bibitem{FAN01}
B.~Fares, P.~Apkarian, and D.~Noll.
\newblock An augmented {L}agrangian method for a class of {LMI}-constrained
  problems in robust control theory.
\newblock {\em International Journal of Control}, 74(4):348--360, 2001.

\bibitem{FKS13}
J.~Fiala, M.~Ko{\v c}vara, and M.~Stingl.
\newblock {PENLAB}: A {MATLAB} solver for nonlinear semidefinite optimization.
\newblock {\em ArXiv e-prints}, November 2013.

\bibitem{For00}
A.~Forsgren.
\newblock Optimality conditions for nonconvex semidefinite programming.
\newblock {\em Mathematical Programming}, 88(1):105--128, 2000.

\bibitem{FSF12}
E.~H. Fukuda, Silva P.~J. S., and M.~Fukushima.
\newblock Differentiable exact penalty functions for nonlinear second-order
  cone programs.
\newblock {\em SIAM Journal on Optimization}, 22(4):1607--1633, 2012.

\bibitem{GP79}
T.~Glad and E.~Polak.
\newblock A multiplier method with automatic limitation of penalty growth.
\newblock {\em Mathematical Programming}, 17(2):140--155, 1979.

\bibitem{Han14}
L.~Han.
\newblock The differentiable exact penalty function for nonlinear semidefinite
  programming.
\newblock {\em Pacific Journal of Optimization}, 10(2):285--303, 2014.

\bibitem{Hig02}
N.~J. Higham.
\newblock Computing the nearest correlation matrix -- a problem from finance.
\newblock 22:329--343, 2002.

\bibitem{Jar12}
F.~Jarre.
\newblock Elementary optimality conditions for nonlinear {SDP}s.
\newblock In {\em Handbook on Semidefinite, Conic and Polynomial Optimization},
  volume 166 of {\em International Series in Operations Research \& Management
  Science}, pages 455--470. Springer, 2012.

\bibitem{KT06}
Y.~Kanno and I.~Takewaki.
\newblock Sequential semidefinite program for maximum robustness design of
  structures under load uncertainty.
\newblock {\em Journal of Optimization Theory and Applications},
  130(2):265--287, 2006.

\bibitem{KKW03}
H.~Konno, N.~Kawadai, and D.~Wu.
\newblock Estimation of failure probability using semi-definite logit model.
\newblock {\em Computational Management Science}, 1(1):59--73, 2003.

\bibitem{KS04}
M.~Ko\v{c}vara and M.~Stingl.
\newblock Solving nonconvex {SDP} problems of structural optimization with
  stability control.
\newblock {\em Optimization Methods \& Software}, 19(5):595--609, 2004.

\bibitem{Lew96}
A.~S. Lewis.
\newblock Convex analysis on the {H}ermitian matrices.
\newblock {\em SIAM Journal on Optimization}, 6(1):164--177, 1996.

\bibitem{LFF16}
B.~F. Louren\c{c}o, E.~H. Fukuda, and M.~Fukushima.
\newblock Optimality conditions for nonlinear semidefinite programming via
  squared slack variables.
\newblock {\em Mathematical Programming}, 168(1):177--200, Mar 2018.

\bibitem{Luc88}
S.~Lucidi.
\newblock New results on a class of exact augmented {L}agrangians.
\newblock {\em Journal of Optimization Theory and Applications}, 58(2), 1988.

\bibitem{N07}
Dominikus Noll.
\newblock Local convergence of an augmented {L}agrangian method for matrix
  inequality constrained programming.
\newblock {\em Optimization Methods and Software}, 22(5):777--802, 2007.

\bibitem{QS11}
Houduo Qi and Defeng Sun.
\newblock An augmented {L}agrangian dual approach for the {H}-weighted nearest
  correlation matrix problem.
\newblock {\em IMA Journal of Numerical Analysis}, 31(2):491--511, 2011.

\bibitem{Sha97}
A.~Shapiro.
\newblock First and second order analysis of nonlinear semidefinite programs.
\newblock {\em Mathematical Programming}, 77(1):301--320, 1997.

\bibitem{SS04}
A.~Shapiro and J.~Sun.
\newblock Some properties of the augmented {L}agrangian in cone constrained
  optimization.
\newblock {\em Mathematics of Operations Research}, 29(3):479--491, 2004.

\bibitem{Tse98}
P.~Tseng.
\newblock Merit functions for semi-definite complementarity problems.
\newblock {\em Mathematical Programming}, 83:159--185, 1998.

\bibitem{YY15}
H.~Yamashita and H.~Yabe.
\newblock A survey of numerical methods for nonlinear semidefinite programming.
\newblock {\em Journal of the Operations Research Society of Japan},
  58(1):24--60, 2015.

\bibitem{YYH12}
H.~Yamashita, H.~Yabe, and K.~Harada.
\newblock A primal-dual interior point method for nonlinear semidefinite
  programming.
\newblock {\em Mathematical Programming}, 135(1-2):89--121, 2012.

\end{thebibliography}

\end{document}